\documentclass[11pt]{article} 
\usepackage[T1]{fontenc}
\usepackage{amsmath}
\usepackage{amssymb}
\usepackage{amsthm}
\usepackage{amsfonts}
\usepackage{dsfont}
\usepackage[mathscr]{euscript}
\usepackage{stmaryrd}
\usepackage{hyperref}
\usepackage{empheq}
\usepackage{cleveref}
\usepackage{fullpage}
\usepackage{enumitem}
\usepackage{color}

\title{Using Sinkhorn in the JKO scheme adds linear diffusion}
\author{Aymeric \textsc{Baradat}$^*$, \and Anastasiia \textsc{Hraivoronska}$^*$, \and Filippo \textsc{Santambrogio}\thanks{Universite Claude Bernard Lyon 1, CNRS, Centrale Lyon, INSA Lyon, Universit\'e Jean Monnet, ICJ UMR5208, 43 bd du 11 Noembre 1918, 69622
Villeurbanne, France, {\tt $\{$baradat,hraivoronska,santambrogio$\}$@math.univ-lyon1.fr}}
}

\date{\today}

\theoremstyle{plain}
\newtheorem{Thm}{Theorem}[section]
\newtheorem{Cor}[Thm]{Corollary}
\newtheorem{Prop}[Thm]{Proposition}
\newtheorem{Lem}[Thm]{Lemma}
\theoremstyle{definition}
\newtheorem{Def}[Thm]{Definition}
\newtheorem{Ass}[Thm]{Assumption}
\theoremstyle{remark}
\newtheorem{Rem}[Thm]{Remark}

\DeclareFontEncoding{LS1}{}{}
\DeclareFontSubstitution{LS1}{stix}{m}{n}
\DeclareSymbolFont{stixletters}{LS1}{stix}{m}{it}
\DeclareMathAccent{\cev}{\mathord}{stixletters}{"91}
\DeclareMathAccent{\vec}{\mathord}{stixletters}{"92}
\DeclareMathAccent{\vecev}{\mathord}{stixletters}{"95}

\newcommand{\R}{\mathbb{R}}
\newcommand{\Z}{\mathbb{Z}}
\newcommand{\T}{\mathbb{T}}

\newcommand{\N}{\mathbb{N}}

\newcommand{\cg}{\langle}
\newcommand{\cd}{\rangle}
\newcommand{\1}{\mathds{1}}
\newcommand{\eps}{\varepsilon}
\renewcommand{\P}{\mathcal P}

\DeclareMathOperator{\Leb}{Leb}
\DeclareMathOperator{\Div}{div}
\DeclareMathOperator{\bDiv}{\mathbf{div}}

\DeclareMathOperator{\tr}{tr}

\DeclareMathOperator{\D}{d \!}
\DeclareMathOperator{\DD}{D^2\!}

\DeclareMathOperator*{\argmin}{arg\,min}

\begin{document} 
	\maketitle
	\begin{abstract}
		The JKO scheme is a time-discrete scheme of implicit Euler type that allows to construct weak solutions of evolution PDEs which have a Wasserstein gradient structure. The purpose of this work is to study the effect of replacing the classical quadratic optimal transport problem by the Schr\"odinger problem (\emph{a.k.a.}\ the entropic regularization of optimal transport, efficiently computed by the Sinkhorn algorithm) at each step of this scheme. We find that if $\eps$ is the regularization parameter of the Schr\"odinger problem, and $\tau$ is the time step parameter, considering the limit $\tau,\eps \to 0$ with $\frac{\eps}{\tau} \to \alpha \in \R_+$ results in adding the term $\frac{\alpha}{2} \Delta \rho$ on the right-hand side of the limiting PDE. In the case $\alpha = 0$ we improve a previous result by Carlier, Duval, Peyré and Schmitzer (2017).
	\end{abstract}

	\section{Introduction}
	
	\paragraph*{The JKO scheme.} The Jordan-Kinderlehrer-Otto scheme ({\it a.k.a.}\ JKO scheme), introduced for the first time in~\cite{jordan1998variational} for linear diffusion equations (Fokker-Planck), then applied in \cite{Otto31012001} to the case of non-linear diffusions, has been studied in a wide variety of contexts since then. We refer to \cite{amb06} and \cite{santambrogio2017overview} for a more detailed presentation. We can say that the JKO scheme is one of the most beautiful applications of optimal transport, as it allows us to solve Cauchy problems -- and hence to predict later states of a system -- using as a brick the Monge-Kantorovich problem (which involves the optimal interpolation between the configurations at two fixed times). 
	
	The idea is that if $\Omega$ is a bounded domain in the Euclidean space or a Riemannian manifold (say for simplicity compact, without boundaries), optimal transport is a way to endow the space of probability measures $\P(\Omega)$ with a formal Riemannian structure. Therefore, given an energy functional $\mathcal F : \P(\Omega) \to \R$, it is possible to wonder what is the formal \emph{gradient flow} of $\mathcal F$ in the space of probability measures. Formally, this gradient flow, that we could write:
	\begin{equation*}
	\dot \rho = - \nabla_{\mathcal P(\Omega)} \mathcal F (\rho),
	\end{equation*}
	 takes the form of the PDE
	\begin{equation}
	\label{eq:gradient_flow}
	\partial_t \rho - \Div \left( \rho \nabla \left(\frac{\delta \mathcal F}{\delta \rho}[\rho]\right) \right) = 0,
	\end{equation}
	where $\frac{\delta \mathcal F}{\delta \rho}[\rho]$ is the first variation of $\mathcal F$ that is the function implicitly defined by the formal condition
	\begin{equation*}
	\frac{\D }{\D s}\mathcal F(\rho + s \varphi)\Big|_{s=0} = \int_\Omega \frac{\delta \mathcal F}{\delta \rho}[\rho](x) \varphi(x) \D x, \qquad \forall \varphi \in \mathcal C(\Omega) \mbox{ with } \int \varphi(x)\D x = 0.
	\end{equation*} 
	As a consequence, many well-known PDEs, such as for instance the heat equation, are formally gradient flows of well-chosen energy functionals (the Boltzmann entropy in the latter case). Therefore, a very natural idea consists in trying to solve these so-called \emph{Wasserstein gradient flows} using an implicit Euler scheme. 
	
	In this context, once $\Omega$ and $\mathcal F$ are chosen, given a time step parameter $\tau>0$ and an initial condition $\rho_0\in \P(\Omega)$, the JKO scheme is:
	\begin{equation}
	\label{eq:JKO}
	\left\{ 
	\begin{gathered}
	\rho_0^\tau \coloneq  \rho_0,\\
	\forall n \in \N, \, \rho^\tau_{n+1} \in \argmin \left\{ \frac{D(\rho^\tau_n, \rho)^2}{2\tau} + \mathcal F(\rho) \, \bigg| \, \rho \in \P(\Omega) \right\},
	\end{gathered}
	\right.
	\end{equation}
	where $D$ stands for the $2$-Wasserstein distance, that we will define below at~Definition~\ref{def:MK}. The goal of the theory is hence to prove that under some assumptions on $\mathcal F$, if $(\rho^\tau_n)_{n \in \N}$ satisfies the induction relation~\eqref{eq:JKO}, the family of piecewise constant curves defined by
	\begin{equation*}
		\bar \rho^\tau(t) \coloneq  \rho^\tau_{\lceil t/\tau \rceil}
	\end{equation*}
	is compact in the set of curves valued in $\mathcal P(\Omega)$ as $\tau \to 0$ (i.e. there exists $\tau_k\to 0$ such that the corresponding curves uniformly converge - for the distance $D$ - to a limit curve), and that the corresponding limit points are weak solution of the PDE~\eqref{eq:gradient_flow}. Once again, we refer to~\cite{santambrogio2017overview} for a review of the available assumptions and techniques allowing to prove such results.

	\paragraph*{Our contribution.} In this work, we study the effect of replacing the usual Wasserstein distance in this scheme by its so-called \emph{entropic regularization}, \emph{a.k.a.}\ \emph{Schr\"odinger cost} studied for instance in~\cite{leonard2012schrodinger,leonard2013survey}. This quantity is labeled by a positive parameter $\eps$, and we denote it by~$D_\eps$. We will define it properly below at Definition~\ref{def:Sch}, but let us already mention that it is obtained from the Wasserstein distance either by adding to a transport cost a penalization of the form $\eps$ times the Boltzmann entropy of the transport plan (see formula~\eqref{def:D_eps_kanto} for more precisions), or equivalently by adding a diffusive term of the form $\frac{\eps}{2} \Delta \rho$ in the right-hand side of the continuity equation in the Benamou-Brenier formulation of optimal transport (which is the formulation that we will use). As we will explain in Subsection~\ref{subsec:sinkhorn}, it is a good approximation of the Wasserstein distance since $D_\eps \to D$ as $\eps \to 0$, and it is very useful for numerics since it can be computed using the so-called \emph{Sinkhorn} algorithm. 
	
	Our main result, precisely stated in Theorem~\ref{thm:main}, can be described as follows. For given $\tau,\eps >0$, let $(\rho^{\tau,\eps}_n)_{n \in \N}$ satisfy the following iterative relations (the \emph{entropic} JKO scheme):
		\begin{equation}
	\label{eq:entropic_JKO}
	\left\{ 
	\begin{gathered}
	\rho_0^{\tau,\eps} \coloneq  \rho_0,\\
	\forall n \in \N, \, \rho^{\tau,\eps}_{n+1} \in \argmin \left\{ \frac{D_{\eps}(\rho^{\tau,\eps}_n, \rho)^2}{2\tau} + \mathcal F(\rho) \, \bigg| \, \rho \in \P(\Omega) \right\}.
	\end{gathered}
	\right.
	\end{equation}
	Then, let us choose $(\tau_k)$ and $(\eps_k)$ two sequences of positive numbers converging to $0$, and such that the ratio $\eps_k / \tau_k$ converges to some $\alpha \in \R_+$ as $k \to + \infty$ (note that we allow $\alpha$ to vanish). Then, under rather usual assumptions for $\mathcal F$ that will be detailed in Subsection~\ref{subsec:ass}, the family of piecewise constant curves 
	\begin{equation*}
	\bar \rho^k(t) \coloneq  \rho^{\tau_k,\eps_k}_{\lceil t/\tau_k \rceil}
	\end{equation*}
	is compact in the set of continuous curves valued in $\mathcal P(\Omega)$ as $k \to + \infty$, and the corresponding limit points are weak solutions of the PDE
	\begin{equation}
    \label{eq:regularized_GF}
	\partial_t \rho - \Div \left( \rho \nabla \left(\frac{\delta \mathcal F}{\delta \rho}[\rho]\right) \right) = \frac{\alpha}{2} \Delta \rho.
	\end{equation}
	
	In words, we identify the correct regime between $\eps$ and $\tau$ to see the effect of the regularization at the limit $\tau \to 0$, and we show that this effect is the addition of linear diffusion in the limiting PDE.
	
	\bigskip
    
    Using $D_\eps$ instead of $D$ as a way of approximating the JKO scheme is not new. It was, for instance, done in \cite{peyre2015entropic}, with convergence results proven in \cite{carlier2017convergence}. But in all the references we know so far, the entropic regularization was just considered as a technical, computational tool to improve the numerical feasibility, and the idea was that $\eps$ had to be very small. This also impacted the quality of the numerical results, as the Sinkhorn algorithm is much less stable and converges slowly for small $\eps$. We remark that even though the case $\eps/\tau\to 0$ is not the main focus of this paper, our result covers this case and holds without the technical assumption $\varepsilon |\log \varepsilon| \leq C \tau^2$ that was present in \cite{carlier2017convergence}.
    
    What we prove here is that the entropic approximation in the term $D_\eps$ provides linear diffusion in the limit. An early result studying this connection appeared in \cite{adams2011large} in the quite restrictive setting corresponding in our language to $\mathcal F = 0$, in one dimension. The validity of the same result in more general setting was somehow conjectured because adding to the functional of the gradient flow a Boltzmann entropy means adding linear diffusion to the PDE, but the delicate point is that the Kantorovich formulation of $D_\eps$ involves the entropy of the \emph{transport plan}, and not of its marginals, see formula~\eqref{def:D_eps_kanto}. 
		
    We also note that some recent results exist about the use of $D_\eps$ in gradient flows when $\eps$ is not small at all: this is what is done in \cite{Hugo-1} and \cite{Hugo-2}, using the so-called {\it debiased Sinkhorn divergence}, which consists in defining a distance using 
        \begin{equation*}
        (\mu,\nu) \in \P(\Omega) \times \P(\Omega) \mapsto D_\eps^2(\mu,\nu)-\frac 12 D_\eps^2(\mu,\mu)-\frac 12 D_\eps^2(\nu,\nu),
        \end{equation*}
        a quantity which contrarily to $D_\eps$ vanishes if $\mu=\nu$ and which is a convex, smooth, positive definite loss function that metrizes the weak convergence of probability measures on a compact set, see \cite{Ramdas2017,Genevay2018,Feydy2018}. However, the spirit of \cite{Hugo-1} and \cite{Hugo-2} is quite different from what we do in this paper, as the goal in \cite{Hugo-2} is to obtain a new PDE, corresponding to the differential structure of the space of probability measures endowed with the geometry of debiased Sinkhorn divergence studied in \cite{Hugo-1}, for fixed $\eps>0$.

\paragraph{Context and notations.}
In this article, we will always work in the flat torus, that is, $\Omega = \T^d$. Here are a few notations that we will constantly use.
\begin{itemize}
	\item We denote by $\Leb$ the Lebesgue measure on $\T^d$, normalized so that 
	$\Leb(\T^d) = 1$.
	\item 	We call $(\sigma_t)_t \geq 0$ the heat kernel on the torus, that is, the solution of
	\begin{equation*}
	\left\{  
	\begin{aligned}
	\partial_t \sigma_t &= \frac{1}{2} \Delta \sigma_t,\\
	\sigma_0 &= \delta_0.
	\end{aligned}
	\right.
	\end{equation*}
	\item Whenever we consider a function $f = f(t,x)$ depending on time and space, $f(t)$ stands for the map $x \mapsto f(t,x)$ depending on space only.
\end{itemize}

\paragraph{Outline of the paper.} 
	In Section~\ref{sec:main_result}, we define properly the Wasserstein distance and the Schr\"odinger cost, explain the link between the two and the role of the Sinkhorn algorithm to compute the second, present the assumptions we use and state the main result of this work. In Section~\ref{sec:sch}, we provide the properties of the Schr\"odinger cost that we need in the proof of Theorem~\ref{thm:main}. Most of the results included in this section are classical, and we will often only provide references to the relevant literature for proofs, but we also include some formal proofs to ensure the text remains self-contained and accessible. In Section~\ref{sec:proof}, we detail the proof of our main result, Theorem~\ref{thm:main}.

\section{Key notions and main result}
	\label{sec:main_result}
	\subsection{The Wasserstein distance and its entropic regularization}
	
	We have not yet described what is the Wasserstein distance we use. There are several ways to define it, but in the present article, we will mainly work in so-called \emph{Benamou-Brenier} formulations~\cite{ben00}. According to this choice, the Wasserstein distance is defined as follows. See~\cite[Section~4.1]{figalli2021invitation} for a more detailed presentation.
	\begin{Def}
		\label{def:MK}
		Let $\mu,\nu \in \P(\T^d)$. The Wasserstein distance between $\mu$ and $\nu$ is defined by
		\begin{equation*}
		\frac{D(\mu,\nu)^2}{2}\coloneq  \inf \left\{ \int_0^1\hspace{-5pt} \int \frac{|v|^2}{2}\rho \D t \, \bigg| \, \begin{gathered}
		\partial_t \rho + \Div(\rho v) = 0, \\  \rho_0 = \mu, \, \rho_1 = \nu  
		\end{gathered}\right\}.
		\end{equation*}
	\end{Def}
	\begin{Rem}
		  This quantity is indeed a distance, and it will actually be our reference distance for $\P(\T^d)$. Namely, when dealing with $\P(\T^d)$, we always consider it as a metric space endowed with $D$. Note that the topology induced by $D$ is the one of weak convergence (that is, in duality with the set of continuous functions).
		
            A rescaling in time in the continuity equation ensures the following equivalent formula which depends on a parameter $\tau>0$ and which is often useful when studying the JKO scheme:
			\begin{equation*}
			\frac{D(\mu,\nu)^2}{2\tau}= \inf \left\{ \int_0^\tau\hspace{-5pt} \int \frac{|v|^2}{2}\rho \D t \, \bigg| \, \begin{gathered}
			\partial_t \rho + \Div(\rho v) = 0, \\  \rho_0 = \mu, \, \rho_\tau = \nu  
			\end{gathered} \right\}.
			\end{equation*}
	\end{Rem}
	
	In this work, our goal is to study the JKO scheme when we replace this Wasserstein distance by an entropic regularized version of it, implying the Schr\"odinger problem instead of the quadratic optimal transport problem. In Benamou-Brenier formulation, this entropic regularization is defined as follows.
	\begin{Def}
		\label{def:Sch}
		Let $\mu,\nu \in \P(\T^d)$ and $\eps>0$. The Schr\"odinger cost of parameter $\eps$ between $\mu$ and~$\nu$ is
		\begin{equation}
		\label{eq:def_Sch}
		\frac{D_\eps(\mu,\nu)^2}{2}\coloneq  \inf \left\{\eps H(\mu) + \int_0^1\hspace{-5pt} \int \frac{|v|^2}{2}\rho \D t \, \bigg| \, \begin{gathered}
		\partial_t \rho + \Div(\rho v) = \frac{\eps}{2} \Delta \rho, \\  \rho_0 = \mu, \, \rho_1 = \nu  
		\end{gathered} \right\},
		\end{equation}
		where $H : \P(\T^d) \to [0,+ \infty]$ is the Boltzmann entropy, defined via 
		\begin{equation*}
		H(\mu)\coloneq \left\{ \begin{aligned}
		&\int_{\T^d} \log\left(\mu(x)\right) \mu(x) \D x && \mbox{if }\mu \ll \Leb, \\
		&+\infty && \mbox{otherwise},
		\end{aligned}\right.
		\end{equation*}
		where we identified measures with their densities with respect to $\Leb$. As $\Leb(\T^d) = 1$, this quantity is always nonnegative.
	\end{Def}
	\begin{Rem}
		  We call the quantity $D_\eps$ a cost and not a distance because it is not a distance. For instance, $D_\eps(\mu,\mu) \neq 0$, unless $\mu = \Leb$. On the other hand, the Schr\"odinger cost is symmetric as we will recall in Corollary~\ref{cor:sym}. Obtaining this convenient symmetry is the reason for the term $\eps H(\mu)$ in~\eqref{eq:def_Sch}.
		
            There are several ways to define the Schr\"odinger cost between $\mu$ and $\nu$, even once chosen to work in Benamou-Brenier formulations. In our work, the formulation above is the most useful one, at least for the formal computations. Yet, we will recall in Proposition~\ref{prop:equivalence} below other equivalent formulations.
		
            Once again, a time rescaling provides the following equivalent formulation for the Schr\"odinger cost:
			\begin{equation}
			\label{eq:rescaled_sch}
			\frac{D_\eps(\mu,\nu)^2}{2\tau}= \inf \left\{\frac{\eps}{\tau} H(\mu) +  \int_0^\tau\hspace{-5pt} \int \frac{|v|^2}{2}\rho \D t \, \bigg| \, \begin{gathered}
			\partial_t \rho + \Div(\rho v) = \frac{\eps}{2\tau} \Delta \rho, \\  \rho_0 = \mu, \, \rho_\tau = \nu  
			\end{gathered}\right\}.
			\end{equation}
                We can already see why formulation~\eqref{eq:def_Sch} is so practical for proving our result: the linear diffusion added in the right hand side of equation~\eqref{eq:regularized_GF} is precisely the one added to the continuity equation in~\eqref{eq:rescaled_sch}.

	\end{Rem}
	
\subsection{Static and dynamic entropic regularization of optimal transport problems}
		
		\label{subsec:sinkhorn}
		
		The goal of this section is to explain the interest in using the Schr\"odinger cost in place of the Wasserstein distance. The main interest is indeed computational.
		
		First of all, as observed in \cite{gentil2017analogy} (see also \cite{gigli2019benamou} for a more abstract result), the dynamical formulation of the Schr\"odinger problem and the associated cost $D_\eps$ is the Benamou-Brenier version of the problem of minimizing the relative entropy of the law of a stochastic process with given initial and final marginals $\mu$ and $\nu$ with respect to the law of the \emph{reversible} Brownian motion, that is, whose initial law is the Lebesgue measure. In other terms, following the celebrated Sanov theorem~\cite{sanov1958probability}, the Schr\"odinger problem can be interpreted as a large deviation problem~\cite{leonard2013survey}. 
		
		Also, a short reasoning involving the additivity of the logarithm implies that the latter problem can be reformulated in terms of a ``static'' entropy minimization. Let us consider all probability measures $\gamma\in\P(\Omega\times\Omega)$ whose marginal measures are given by $(\pi_x)_\#\gamma=\mu$ and $(\pi_y)_\#\gamma=\nu$. These probability measures on the product space are called {\it transport plans} in the Monge-Kantorovich theory and the set of transport plans is often denoted by $\Pi(\mu,\nu)$. The reformulation in \cite{gentil2017analogy} writes
		\begin{equation*}
		\frac{1}{2}D_\eps^2(\mu,\nu)=\min\{\eps H(\gamma|R_\eps)\,:\,\gamma\in \Pi(\mu,\nu)\},
		\end{equation*}
		where $H(\cdot|\cdot)$ denotes the relative entropy and $R_\eps$ is the joint law of a random variable of the form $(B_0,B_\eps)$, $(B_t)_{t \geq 0}$ being a reversible Brownian motion. Knowing the transition kernel of the Brownian motion, let us define the cost $c_\eps(x,y)$ through the formula:
		\begin{equation*}
		\frac{\D R_\eps}{\D x \otimes \D y}(x,y) = \sigma_\eps(y-x) =: \frac{1}{\sqrt{2 \pi \eps}^d}\exp\left( -\frac{c_\eps(x,y)}{\eps} \right).
		\end{equation*}
		Then, it is possible to see that our problem can be reformulated as follows:
		\begin{equation}
		\label{def:D_eps_kanto}
		\frac{1}{2} D_\eps^2(\mu,\nu)=\min\left\{\int c_\eps d\gamma +\eps H(\gamma) + \frac{d \eps}{2} \log (2 \pi \eps) \,:\,\gamma\in \Pi(\mu,\nu)\right\},
		\end{equation}
		where the entroy is now a Boltzmann entropy for measures on the product space. This is the so-called {\it entropic regularization} of the optimal transport problem with cost $c_\eps$, as forgetting about the constant term which has an effect on the optimal value but not on the optimizers, it consists in adding a penalization of the form of an entropy to an optimal transport problem.

        \bigskip
		
		Solving problems of this type can be attacked numerically via the Sinkhorn algorithm, also called {\it  iterative proportional fitting procedure} (IPFP), which can be seen as an alternate maximization on the dual problem or as alternate projections for the Kullback-Leibler divergence in terms of the primal one. This was introduced long ago in 
		~\cite{sinkhorn1964relationship,sinkhorn1967diagonal}, and then popularized in optimal transport in~\cite{cuturi2013sinkhorn} and~\cite{benamou2015iterative}. This iterative algorithm is incredibly efficient, especially when $\eps$ is not too small.

        \bigskip
        
		It is interesting to observe that, if we replace the torus $\T^d$ with the Euclidean space~$\R^d$ (which leads to problems since the Lebesgue measures is not a probability measure on $\R^d$, but let us not comment on this technical issue), then we obtain the following explicit expression for $c_\eps$, independent of~$\eps$: 
		\begin{equation}
		\label{eq:euclidean_cost}
		c_\eps(x,y)=\frac{|x-y|^2}{2}.
		\end{equation}
		This makes it clear that, in the limit $\eps\to 0$, the computed value should converge to the squared Wasserstein distance (in its Kantorovich formulation, and not in the dynamic formulation that we presented here). This is rigorously established in~\cite{carlier2017convergence} under the necessary and sufficient condition that $H(\mu)$ and $H(\nu)$ are finite (otherwise, $D_\eps$ is infinite whatever $\eps>0$).
		
		In the case of the torus, identity~\eqref{eq:euclidean_cost} does not hold anymore, but is replaced by
		\begin{equation}
		\label{eq:torus_cost}
		c_\eps(x,y) = -\eps \log\left( \sum_{k \in \Z^d} \exp\left( -\frac{|Y+k - X|^2}{2\eps} \right) \right),
		\end{equation}
		where $X$ and $Y$ are any lifts of $x$ and $y$ from $\T^d$ to $\R^d$. Still:
		\begin{itemize}
			\item The convergence of $D_\eps$ towards $D$ at the level of the Kantorovich formulation can be proved using the ideas of~\cite{carlier2017convergence} in virtue of the uniform convergence
			\begin{equation*}
			c_\eps(x,y) \underset{\eps \to 0}{\longrightarrow} \frac{d(x,y)^2}{2},
			\end{equation*}
			where $d$ is the canonical distance on $\T^d$.
			\item The Sinkhorn algorithm can be used efficiently to compute $D_\eps$, as this efficiency is not affected by replacing the squared distance with $c_\eps$. The only difference with the standard case is that one needs to compute $c_\eps$ once for all, either by using formula~\eqref{eq:torus_cost}, or by solving the heat equation. 
			\item We believe that our result remains the same replacing $D_\eps$ by the quantity~\eqref{def:D_eps_kanto} in which $c_\eps$ is replaced by $d(x,y)^2/2$, but we did not try to prove this fact. Once our result is known, such a proof would reduce to establishing some quantitative stability in the entropic regularization of optimal transport with respect to the cost.
		\end{itemize}
		The same could also be done in bounded domains using reflected Brownian motions and the heat kernel with Neumann boundary conditions, but we decided not to insist in this paper on the boundary issues (in general, we should impose everywhere no-flux boundary conditions).
		
		We close this section by noticing that the proof of convergence $D_\eps\to D$ under the condition that $H(\mu)$ and $H(\nu)$ are finite is also possible to perform at the level of the dynamical formulations~\cite{leonard2012schrodinger,baradat2020small}.

		With all these observations, we believe that it is natural to study the effect of replacing $D$ by its regularized version in the JKO scheme, when letting both $\eps$ and $\tau$ go to $0$ in an appropriate regime.

	\subsection{Precise statement of the main result} 
		\label{subsec:ass}
	In this article, we consider energy functionals $\mathcal F$ of the form 
    \begin{equation}
	\label{eq:def_F}
	\mathcal F(\rho) \coloneq  \left\{   
	\begin{aligned}
	&\int \Big\{ V(x)\rho(x) + \frac{1}{2}W \ast \rho(x)  \rho(x) + f(\rho(x)) \Big\}\D x, &&\mbox{if }\rho \ll \Leb, \\
	&+ \infty,&&\mbox{otherwise},
	\end{aligned}
	\right.
	\end{equation}
	where $V,W: \T^d \to \R$ and $f: \R_+ \to \R$ satisfy assumptions to be specified in a few lines. At the formal level, it can be checked that for a given $\rho \in \P(\T^d)$ with $\rho \ll \Leb$,
\begin{equation*}
    \frac{\delta \mathcal F}{\delta \rho}[\rho] = V + W \ast \rho + f'(\rho),
\end{equation*}
    so that \eqref{eq:gradient_flow} writes
    \begin{equation*}
	\partial_t \rho - \Div ( \rho (\nabla V + \nabla W \ast \rho )) = \Delta g(\rho),
	\end{equation*}
	where $g : \R_+ \to \R$ is the function defined for all $s \in \R_+$ by
	\begin{equation}
	\label{eq:def_g}
	g(s) = sf'(s) - f(s).
	\end{equation}
    Of course, equation~\eqref{eq:regularized_GF} then becomes
    	\begin{equation}
	\label{eq:limit_PDE}
	\partial_t \rho - \Div ( \rho (\nabla V + \nabla W \ast \rho )) = \Delta g(\rho) + \frac{\alpha}{2} \Delta \rho.
	\end{equation}

	\subsubsection{Assumptions on the internal energy potential}
	
	To state our main result, we will need some assumptions on the function $f$ appearing in the definition of $\mathcal F$ in~\eqref{eq:def_F}. Our goal is to consider rather general functions $f$, imposing conditions on its smoothness and growth, but not relying on some specific kind of nonlinearity. The general assumptions we choose are clearly not optimal (see Remark~\ref{rem:assumption_nonlinearity}) but lead to rather direct proofs. Two different assumptions will appear at different stages of the proof of Theorem~\ref{thm:main}. 
	
	The first one is necessary for proving that for positive~$\tau,\eps$, equation~\eqref{eq:limit_PDE} is satisfied up to a term converging to $0$ as $\tau,\eps \to 0$. The second one concerns strong compactness and is therefore necessary to prove that the limiting curve is a distributional solution of the (nonlinear) limiting PDE~\eqref{eq:limit_PDE}.
	
	\begin{Ass}
		\label{ass:smooth_f}
		The function $f: \R_+ \to \R \cup \{+\infty\}$ is convex, and the domain where it is finite $\mathrm{Dom}(f)\coloneq \{f<+\infty\}$ is a nonnempty and nontrivial interval of $\R_+$. In other words, calling $d_-,d_+ \in \R_+ \cup \{ + \infty \}$ the lower and upper bound of $\mathrm{Dom}(f)$, we have $d_- < d_+$. Moreover, we assume that $f$ is continuous on $\mathrm{Dom}(f)$, of class $C^3$ on its interior $(d_-,d_+)$, and that
		\begin{equation}
		\label{eq:ass_f_smooth}
		\mbox{if } d_- > 0, \mbox{ then } f'(s) \underset{s \to d_-}{\longrightarrow} - \infty \qquad \mbox{and}\qquad \mbox{if } d_+ < +\infty, \mbox{ then } f'(s) \underset{s \to d_+}{\longrightarrow} + \infty.
		\end{equation}
		Finally, we assume that $d_- < 1 < d_+$.
	\end{Ass}
	
	\begin{Rem}
		\label{rem:assumption_f_smooth}
        Under Assumption~\ref{ass:smooth_f}, the function $g$ defined by~\eqref{eq:def_g} is well defined on $(d_-,d_+)$. As, it is increasing, we can extend it with potentially infinite values at $d_-$ and $d_+$.
			
            The Assumption $d_- < 1 < d_+$ is very natural. Indeed, in this paper, we consider measures $\rho \in \P(\T^d)$ such that $\mathcal F(\rho)<+\infty$, that is, measures with densities satisfying
			\begin{equation*}
			\int f(\rho) \D x < + \infty.
			\end{equation*}
			But by Jensen's inequality, if $f$ is convex,
			\begin{equation*}
			\int f(\rho) \geq f(1).
			\end{equation*}
			So in particular, for such a $\rho$ to exist, $1 \in \mathrm{Dom}(f) \subset [d_-,d_+]$. In addition, under condition~\eqref{eq:ass_f_smooth}, it is easy to see that if $1 \in \{d_-,d_+\}$, then the Lebesgue measure is the only measure for which $\mathcal F$ is finite. We do not want to work in this situation, and therefore we assume $d_- < 1 < d_+$.
	\end{Rem}
	Our second assumption is the following.
	\begin{Ass}
		\label{ass:strong_compactness}
		Either $\alpha>0$ and we have nothing more to assume. Or $\alpha = 0$, and then, we need two assumptions. First, we need the following growth condition: either the upper bound $d_+$ is finite, or the function $h$ defined for all $s \in (d_-, d_+)$ by
			\begin{equation*}
			h(s) \coloneq  \int_1^s \sqrt r f''(r) \D r,
			\end{equation*}
			satisfies
			\begin{equation}
			\label{eq:growth_condition}
			\frac{h(s)}{\displaystyle s^{\frac{d-2}{2d}}} \underset{s \to + \infty}{\longrightarrow} + \infty.
			\end{equation}
			
			Second, we assume that at least one of the following conditions holds
			\begin{enumerate}[label=(\roman*)]
				\item $f$ is superlinear,
				\item $d_- = 0$,
				\item $f(d_-) = +\infty$.
			\end{enumerate}
			
	\end{Ass}
	\begin{Rem}
		\label{rem:assumption_nonlinearity}
	
			The fact that we do not need any further assumption in the case when $\alpha>0$ can be understood as follows. In that case, the strong compactness needed to pass to the limit directly comes from the regularity properties of the Schr\"odinger problem. On the other hand, when $\alpha=0$, the strong compactness has to be provided by the limit PDE, and we have to make assumptions similar to the ones of the non-entropic JKO scheme. We chose conditions that provide a rather simple proof in the classical case, but which are not optimal. We believe that any assumption which would make the classical JKO scheme converge would be enough to prove our convergence result with $\alpha = 0$, but we wanted to stay in a rather easy setting in order to emphasize the arguments specific to the entropic case.
			
            A quick computation shows that if $f: s \mapsto s\log s$, then $h(s) = \mathcal{O}_{s \to + \infty}( s^{1/2} )$, so the condition~\eqref{eq:growth_condition} is satisfied. In the case where $f: s \mapsto s^m/(m-1)$, $m \neq 1$, that is the porous medium or fast diffusion case, then $h(s) = \mathcal O_{s \to + \infty} (s^{m - 1/2}) $. The assumption becomes $m> 1 - 1/d$, which is essentially the same as~\cite[Condition~(10.4.21)]{amb06}, used in~\cite[Example~11.2.4]{amb06}. Here, we will provide a straightforward proof.
		
	\end{Rem}
	
	\subsubsection{Precise statement of the main result}
	
	We are now in a position to state our main result.
	
	\begin{Thm}
		\label{thm:main}
		Let us assume that $V$ and $W$ are of class $C^2$, that $f$ satisfies Assumption~\ref{ass:smooth_f}, and let us define $\mathcal F$ as in~\eqref{eq:def_F} and $g$ as in~\eqref{eq:def_g}. Let us consider two sequences of positive numbers $(\tau_k)_{k \in \N}$ and $(\eps_k)_{k \in \N}$ such that
		\begin{equation*}
		\tau_k\underset{k \to + \infty}{\longrightarrow} 0 \qquad \mbox{and} \qquad \alpha_k \coloneq  \frac{\eps_k}{\tau_k}\underset{k \to + \infty}{\longrightarrow} \alpha \in [0,+\infty).
		\end{equation*}		
		Let $\rho_0 \in \P(\T^d)$ satisfy $\mathcal F(\rho_0)< + \infty$ and $H(\rho_0)< + \infty$. Let $(\rho^k_n)_{n \in \N}$ be a sequence satisfying the inductive relation~\eqref{eq:entropic_JKO} with $\eps = \eps_k$ and $\tau = \tau_k$. Finally, let us define for all $t\geq 0$
		\begin{equation*}
		\bar \rho^k(t) \coloneq  \rho^k_{\lceil t/\tau_k \rceil}.
		\end{equation*}
		Then, for any $T>0$, the following assertions hold:
		\begin{itemize}
			\item The sequence $(\bar \rho^k)$ is compact  for the uniform convergence of functions on $[0,T]$ valued into $ \P(\T^d)$, when this set of probabilities is endowed with the distance $D$. Moreover, its limit points belong to $C([0,T]; \P(\T^d))$.
			\item For all $k \in \N$ and $t\in[0,T]$, $\bar \rho^k(t)$ has a density, $g(\bar \rho_k)$ belongs to $L^1([0,T] \times \T^d)$ and the following convergence holds in $\mathcal D'((0,T) \times \T^d)$:
			\begin{equation}
			\label{eq:distributional_convergence}
			\partial_t \bar \rho^k - \Div ( \bar \rho^k (\nabla V + \nabla W \ast \bar \rho^k )) - \Delta g(\bar \rho^k) - \frac{\alpha_k}{2} \Delta \bar \rho^k \underset{k \to + \infty}{\longrightarrow} 0.
			\end{equation}
			\item If in addition, $f$ satisfies Assumption~\ref{ass:strong_compactness}, any limit point $\rho$ of $(\bar \rho^k)$ given by the first point has a density, $g(\rho)$ belongs to $L^1([0,T] \times \T^d)$, and $\rho$ is a distributional solution of 
			\begin{equation}
			\label{eq:elliptic_reg_PDE}
			\left\{ 
			\begin{gathered}
			\partial_t \rho - \Div ( \rho (\nabla V + \nabla W \ast \rho )) = \Delta g(\rho) + \frac{\alpha}{2} \Delta \rho,\\
			\rho|_{t=0} = \rho_0.
			\end{gathered}
			\right.
			\end{equation}
		\end{itemize}	
	\end{Thm}
	
	\begin{Rem}
		  In the case when $\alpha = 0$, that is, when the regularization parameter converges to $0$ faster than the time step, we improve a previous result from~\cite{carlier2017convergence} where the authors had to assume the much stronger condition $\eps_k |\log \eps_k| = \mathcal O_{k \to + \infty} (\tau_k^2)$.
			
            Our result shows that to get a linear diffusion term $\Delta \rho$ in the limiting PDE, we can either act at the level of the energy functional by including some Boltzmann entropy (if $f(s) = s \log s$, then $g(s) = s$), act at the level of the "distance" part of the scheme by keeping $\alpha>0$, or mix both possibilities. As it is easier to compute Schr\"odinger costs than Wasserstein distances, especially for large values of $\eps$, it seems relevant to choose to act at the level of the distance only. Yet, a rigorous comparison of the two approaches would also need to know the convergence rates for both schemes: it is useful to compute each time step more quickly only provided the convergence rate is not dramatically worsened.
		
        The picky reader could observe that the case $\alpha=0$ and $f=0$ is not covered in this theorem. Of course, in that case, the convergence is even easier as strong compactness for $\rho$ is not needed to justify the limiting PDE.
		
	\end{Rem}

	\section{Preliminaries on the Schr\"odinger problem}
	\label{sec:sch}
	In this section, we gather a few standard results about the Schr\"odinger problem that will be useful in the proof of Theorem~\ref{thm:main}. In that proof, there will always be a rescaling by $\tau$ as in~\eqref{eq:rescaled_sch}, so the only somehow unusual aspect of our presentation is that we will carry this rescaling all along to be able to use the results directly later. Of course, to get more standard formulations, it suffices to take $\tau = 1$. We give sketches of proofs when we think it can be useful, but mainly give references to where the results can be found in the literature. Only Proposition~\ref{prop:right_derivative_sch} is mainly useful in the context of JKO, and hence somehow new, although expected. Therefore, for this last proposition, we provide a complete proof.
	
	First, the Schr\"odinger cost is convex and lower semi-continuous in the topology of narrow convergence.
	\begin{Prop}
		\label{prop:sch_convex}
		Let $\eps>0$. The quantities $(\mu,\nu) \mapsto D_\eps^2(\mu,\nu)$  and $(\mu,\nu) \mapsto D_\eps^2(\mu,\nu)/2 - \eps H(\mu)$ are convex and lower semi-continuous in the topology of narrow convergence.
	\end{Prop}
	\begin{proof}[Formal proof]
		This is a direct consequence of the convexity and lower semi-continuity of the cost and linearity of the constraints in~\eqref{eq:def_Sch}, when considered in the variables $(\rho,\rho v)$. See~\cite[Section~2.1]{baradat2020small} for more details.
	\end{proof}
	
	Actually, the Schr\"odinger problem has several formulations, even in the Benamou-Brenier style. The one we presented in Definition~\ref{def:Sch} is forward in time in the sense that the density follows a forward parabolic equation, but there is another formulation that is backward in time, and a third one that is symmetric in time. We refer to~\cite{leonard2013survey,gentil2017analogy,gigli2019benamou} for the rigorous proofs. 
	
	\begin{Prop}
		\label{prop:equivalence}
		Let $\mu,\nu \in \P(\T^d)$ and $\eps,\tau>0$. We have
		\begin{align}
		\label{eq:Sch_backward} \frac{D_\eps(\mu,\nu)^2}{2\tau} &= \inf \left\{\frac{\eps}{\tau} H(\nu) + \int_0^\tau\hspace{-5pt} \int \frac{|w|^2}{2}\rho \D t \, \bigg| \, \begin{gathered}
		\partial_t \rho + \Div(\rho w) = - \frac{\eps}{2\tau} \Delta \rho, \\  \rho_0 = \mu, \, \rho_\tau = \nu  
		\end{gathered}\right\} \\
		\label{eq:Sch_symmetric} &= \inf \left\{\frac{\eps}{\tau} \frac{H(\mu) + H(\nu)}{2} + \frac{1}{2}\int_0^\tau\hspace{-5pt} \int \bigg[ |c|^2 + \left| \frac{\eps}{2\tau} \nabla \log \rho\right|^2\bigg]\rho \D t \, \bigg| \, \begin{gathered}
		\partial_t \rho + \Div(\rho c) = 0, \\  \rho_0 = \mu, \, \rho_\tau = \nu  
		\end{gathered}\right\}.
		\end{align}
	\end{Prop}
	\begin{proof}[Formal proof]
		Let us explain the main idea without checking all the regularity needed to justify the computations. Let us consider $(\rho,v)$ a competitor for~\eqref{eq:rescaled_sch} and call
		\begin{equation*}
		c \coloneq  v - \frac{\eps}{2\tau} \nabla \log \rho, \qquad w \coloneq  v - \frac{\eps}{\tau} \nabla \log \rho.
		\end{equation*} 
		The first observation is that
		\begin{equation}
		\label{eq:continuity_equation}
		\partial_t \rho + \Div(\rho c ) = 0 \qquad \mbox{and}\qquad \partial_t \rho + \Div (\rho w ) = - \frac{\eps}{2\tau} \Delta \rho.
		\end{equation}
		In addition, we have
		\begin{equation*}
		\frac{1}{2}\int_0^\tau \hspace{-5pt} \int |v|^2\rho \D t = \frac{1}{2}\int_0^\tau \hspace{-5pt} \int |c|^2\rho \D t + \frac{1}{2}\int_0^\tau \hspace{-5pt} \int \left|\frac{\eps}{2\tau} \nabla \log \rho\right|^2\rho \D t + \frac{\eps}{2\tau}\int_0^\tau \hspace{-5pt} \int \left\{ c \cdot \nabla \log \rho \right\} \rho \D t.
		\end{equation*}
		On the other hand, equation~\eqref{eq:continuity_equation} implies
		\begin{equation*}
		\frac{\D }{\D t} \int \rho \log \rho = \int \left\{ c \cdot \nabla \log \rho \right\}\rho.
		\end{equation*}
		Therefore,
		\begin{equation*}
		\frac{\eps}{2\tau}\int_0^\tau \hspace{-5pt} \int \left\{ c \cdot \nabla \log \rho \right\} \rho \D t = \frac{\eps}{2\tau} \left( H(\nu) - H(\mu) \right).
		\end{equation*}
		Identity~\eqref{eq:Sch_symmetric} follows, and~\eqref{eq:Sch_backward} is proved in the same way.
	\end{proof}

	As a corollary of this result, we obtain the symmetry of $D_\eps$.
	\begin{Cor}
		\label{cor:sym}
		For any $\mu,\nu \in \P(\T^d)$ and $\eps>0$, $D_\eps(\mu,\nu) = D_\eps(\nu,\mu)$.
	\end{Cor}
	\begin{proof}
		Formulation~\eqref{eq:Sch_symmetric} is clearly invariant through time reversal.
	\end{proof}
	
	The next result is concerned with the existence of optimizers for the minimization problems in play. It can be deduced from the existence results in~\cite{leonard2013survey} or~\cite{baradat2020small}.
	\begin{Prop}
		\label{prop:existence}
		Let $\mu,\nu \in \P(\T^d)$ and $\eps,\tau>0$. The Schr\"odinger cost $D_\eps(\mu,\nu)$ as defined in Definition~\ref{def:Sch} is finite if and only if $H(\mu)$ and $H(\nu)$ are finite. In that case, there exist unique minimizers $(\rho,v)$, $(\rho,c)$ and $(\rho,w)$ achieving the infimum in~\eqref{eq:def_Sch}, \eqref{eq:Sch_backward} and~\eqref{eq:Sch_symmetric} respectively, of common density. In particular 
		\begin{equation}
		\label{eq:formula_actions}
		\frac{\eps}{\tau} H(\mu) +  \int_0^\tau\hspace{-5pt} \int\! \frac{|v|^2}{2}\rho \D t = \frac{\eps}{\tau} \frac{H(\mu) + H(\nu)}{2} + \frac{1}{2} \int_0^\tau\hspace{-5pt} \int \hspace{-3pt}\bigg\{\! |c|^2 + \left| \frac{\eps}{2\tau} \nabla \log \rho\right|^2 \! \bigg\}\rho \D t  = \frac{\eps}{\tau} H(\nu) + \int_0^\tau\hspace{-5pt} \int \!\frac{|w|^2}{2}\rho \D t.
		\end{equation}
		In addition, $\rho$ is in $AC^2([0,\tau]; \P(\T^d))$, with metric derivative
		\begin{equation*}
		|\dot \rho(t)|^2 = \int |c(t)|^2 \rho(t), \qquad \mbox{for a.e.\ }t \in [0,\tau].
		\end{equation*}
		Finally, we have almost everywhere
		\begin{equation}
		\label{eq:link_v_w}
		v - \frac{\eps}{2\tau} \nabla \log \rho = c = w + \frac{\eps}{2\tau} \nabla \log \rho.
		\end{equation}
	\end{Prop}
	
	The next proposition is a duality result, which is a direct application of the results in~\cite{dimarino2019optimal} once the equivalence between static and Benamou-Brenier formulations of the Schr\"odinger problem  -- as presented in~\cite{gentil2017analogy} -- is understood. Here, we will propose a very different formal proof based on duality arguments starting from the Benamou-Brenier formulation.
	
	\begin{Prop}
		\label{prop:duality}
		Let $\mu$, $\nu \in \P(\T^d)$ and $\eps,\tau >0$. Then
		\begin{equation}
		\label{eq:dual_forward}\frac{D_\eps(\mu,\nu)^2}{2\tau} =  \frac{\eps}{\tau} H(\mu) + \sup_{\varphi \in L^1(\nu)} \Big\cg \varphi, \nu \Big\cd - \Big\cg \frac{\eps}{\tau} \log \Big( \left\{ \exp\left(\frac{\tau \varphi}{\eps}\right) \right\}\ast \sigma_{\eps} \Big), \mu \Big\cd.
		\end{equation}
		
		Moreover, as soon as $H(\mu)$ and $H(\nu)$ are finite, the maximization problem on the right-hand side admits a maximizer $\varphi$, and if we call
		\begin{equation}
		\label{eq:psi_from_phi}
		\psi \coloneq  \frac{\eps}{\tau}\log \mu -\frac{\eps}{\tau} \log \Big( \left\{  \exp\left(\frac{\tau \varphi}{\eps}\right) \right\}\ast \sigma_{\eps} \Big),
		\end{equation}
		we have
		\begin{equation}
		\label{eq:phi_from_psi}
		\varphi  = \frac{\eps}{\tau}\log \nu - \frac{\eps}{\tau} \log \Big( \left\{ \exp\left(\frac{\tau \psi}{\eps}\right) \right\}\ast \sigma_{\eps} \Big).
		\end{equation}
		Stated differently, $\varphi$ satisfies the fixed point condition:
		\begin{equation}
		\label{eq:optimality_condition_varphi}
		\varphi = \frac{\eps}{\tau} \log \nu  - \frac{\eps}{\tau}\log \left( \left\{ \frac{\mu}{\exp\left( \frac{\tau}{\eps} \varphi \right)\ast \sigma_\eps} \right\} \ast \sigma_\eps \right).
		\end{equation}
		In particular, $\varphi - \frac{\eps}{\tau} \log \nu$ is of class $C^\infty$.
		
		Finally, there is a unique $\varphi$ satisfying~\eqref{eq:optimality_condition_varphi} up to an additive constant and hence a unique maximizer in~\eqref{eq:dual_forward} up to an additive constant.
	\end{Prop}
	\begin{Rem}
		In~\eqref{eq:dual_forward}, it looks like we broke the symmetry of the problem. This is because as we will see in the sketch of proof, these formulas are obtained by looking at the dual of the Schr\"odinger problem formulated as in~\eqref{eq:rescaled_sch}. If we were starting from~\eqref{eq:Sch_backward}, we would obtain a symmetric dual problem.
	\end{Rem}
	\begin{proof}[Formal proof]
		Let us write problem~\eqref{eq:rescaled_sch} formally as an inf-sup problem, in the spirit of~\cite[Section~1.4]{baradat2021regularized}:
		\begin{align*}
		\frac{D_\eps(\mu,\nu)^2}{2\tau} = \frac{\eps}{\tau}H(\mu) + \inf_{(\rho,v)}\sup_{\varphi}& \int_0^\tau\hspace{-5pt} \int \frac{|v|^2}{2}\rho \D t \\
		&+  \cg \varphi(\tau),\nu \cd - \cg \varphi(0),\mu \cd - \int_0^\tau\cg \partial_t \varphi + v \cdot \nabla \varphi + \frac{\eps}{2\tau} \Delta \varphi, \rho\cd \D t,
		\end{align*}
		where forgetting any regularity issue, the inf is taken over functions $\rho: [0,\tau] \times \T^d \to \R_+$ and $v:[0,\tau] \times \T^d \to \R^d$, and the sup is taken over $\varphi: [0,\tau] \times \T^d \to \R$. As the functional inside the inf-sup is convex in $(\rho,\rho v)$ and concave (actually, affine) in $\varphi$, it seems legitimate to swap the inf and the sup, ending up with:
		\begin{align*}
		\frac{D_\eps(\mu,\nu)^2}{2\tau} = \frac{\eps}{\tau}H(\mu) +\sup_{\varphi}&\, \cg \varphi(\tau),\nu \cd - \cg \varphi(0),\mu \cd  \\
		&+\inf_{(\rho,v)} \int_0^\tau\hspace{-5pt} \int \frac{|v|^2}{2}\rho \D t 
		- \int_0^\tau\cg \partial_t \varphi + v \cdot \nabla \varphi + \frac{\eps}{2\tau} \Delta \varphi, \rho\cd \D t.
		\end{align*}
		Minimizing in $v$ once $\varphi$ is fixed, we find $v = \nabla \varphi$, and then minimizing in $\rho$, we find that the minimizer is $-\infty$ as soon as the inequality $\partial_t \varphi + \frac{1}{2}|\nabla \varphi|^2 + \frac{\eps}{2\tau} \Delta \varphi \leq 0$ does not hold everywhere, and cancels in the opposite case. We conclude that we have
		\begin{equation}
		\label{eq:dual_formula_varphi}
		\frac{D_\eps(\mu,\nu)^2}{2\tau} = \frac{\eps}{\tau}H(\mu) +\sup_{\substack{\varphi \mbox{ \footnotesize s.t.}\\ \partial_t \varphi + \frac{1}{2}|\nabla \varphi|^2 + \frac{\eps}{2\tau} \Delta \varphi \leq 0}} \cg \varphi(\tau),\nu \cd - \cg \varphi(0),\mu \cd .
		\end{equation}
		
		Now, given a competitor $\varphi$ satisfying the constraint, let us call 
		\begin{equation*}
		\theta \coloneq  \exp \left( \frac{\tau \varphi}{\eps} \right) \qquad \mbox{so that} \qquad \partial_t \theta + \frac{\eps}{2\tau} \Delta \theta \leq 0,
		\end{equation*}
		and the problem becomes
		\begin{equation*}
		\frac{D_\eps(\mu,\nu)^2}{2\tau} = \frac{\eps}{\tau}H(\mu) +\sup_{\substack{\theta \mbox{ \footnotesize s.t.}\\ \partial_t \theta + \frac{\eps}{2\tau} \Delta\theta \leq 0}} \left\cg \frac{\eps}{\tau} \log \theta(\tau),\nu \right\cd - \left\cg \frac{\eps}{\tau} \log \theta(0),\mu \right\cd .
		\end{equation*}
		
		Due to the sign in the constraint, given a final condition $\theta(\tau)$, we increase the value of the right-hand side by replacing $\theta(0)$ by $\theta(\tau) \ast \sigma_\eps$. But this change corresponds to considering only the equality case in the constraint, which is still admissible. This means that we do not change the value of the problem by replacing the inequality sign with an equality sign in the constraint, and we deduce
		\begin{align*}
		\frac{D_\eps(\mu,\nu)^2}{2\tau} &= \frac{\eps}{\tau}H(\mu) +\sup_{\theta: \T^d \to \R} \left\cg \frac{\eps}{\tau} \log \theta,\nu \right\cd - \left\cg \frac{\eps}{\tau} \log (\theta \ast \sigma_\eps),\mu \right\cd\notag \\
		&= \frac{\eps}{\tau}H(\mu) +\sup_{\varphi: \T^d \to \R} \left\cg \varphi,\nu \right\cd - \left\cg \frac{\eps}{\tau} \log\left( \exp\left( \frac{\tau\varphi}{\eps} \right) \ast \sigma_\eps\right),\mu \right\cd,
		\end{align*}
		where in the second line, we changed the variable according to $\varphi = \frac{\eps}{\tau}\log \theta$. This leads to~\eqref{eq:dual_forward}.
		
		The final step consists in deriving the optimality condition for $\varphi$ in this last formula. A direct first-order computation gives
		\begin{equation*}
		\nu  =  \exp\left( \frac{\tau}{\eps} \varphi \right) \left( \frac{\mu}{\exp\left( \frac{\tau}{\eps} \varphi \right)\ast \sigma_\eps} \right) \ast \sigma_\eps,
		\end{equation*}
		which implies~\eqref{eq:optimality_condition_varphi} by rearranging the terms. In particular, defining $\psi$ as in~\eqref{eq:psi_from_phi}, we get~\eqref{eq:phi_from_psi}. The regularity result follows directly from these formulas.
		
		Uniqueness is obtained by showing that~\eqref{eq:optimality_condition_varphi} admits a unique solution up to an additive constant, and we will not comment on that issue further.
	\end{proof}
	
	In the next proposition, we explicit the link between the so-called Schr\"odinger potentials $\varphi$ and $\psi$ obtained in Proposition~\ref{prop:duality}, and the optimal pairs $(\rho,v)$ and $(\rho,w)$ from Proposition~\ref{prop:existence}. These properties can be found in~\cite{leonard2013survey}.
	\begin{Prop}
		\label{prop:link_duality_BB}
		Let $\mu$, $\nu \in \P(\T^d)$ with $H(\mu)$ and $H(\nu)$ finite, and $\eps,\tau >0$. Let $(\rho,v)$ and $(\rho,w)$ be the minimizers from~\eqref{eq:rescaled_sch} and~\eqref{eq:Sch_backward} given by Proposition~\ref{prop:existence}. Let $\bar\varphi$ be a maximizer of~\eqref{eq:dual_forward} given by Proposition~\ref{prop:duality} and $\bar\psi$ defined as in~\eqref{eq:psi_from_phi}. For any $t \in [0,\tau]$, call
		\begin{equation}
		\label{eq:explicit_varphi_psi}
		\varphi(t) \coloneq  \frac{\eps}{\tau} \log \Big( \exp\left(  \frac{\tau \bar\varphi}{\eps}\right) \ast \sigma_{\frac{\eps }{\tau}(\tau - t)} \Big) \qquad \mbox{and} \qquad \psi(t) \coloneq  \frac{\eps}{\tau} \log \Big( \exp\left(  \frac{\tau \bar\psi}{\eps}\right) \ast \sigma_{\frac{\eps }{\tau}t} \Big).
		\end{equation}
		Then, we have
		\begin{align}
		\label{eq:v_is_nabla_varphi} \mbox{for almost all }t \in [0,\tau],&& v(t) = \nabla \varphi(t) \qquad &\mbox{and} \qquad w(t) = \nabla \psi(t) \\
		\label{eq:explicit_rho}\mbox{for all }t \in [0,\tau],&& \rho(t) = \exp&\left( \frac{\tau}{\eps}(\varphi(t) - \psi(t))  \right).
		\end{align}
		In particular, up to modifying $v$ and $w$ on a negligible set of times, $\rho$, $v$ and $w$ are of class $C^\infty$ on $(0,\tau) \times \T^d$.
		
		Finally, 
		\begin{equation}
		\label{eq:link_a_varphi}
		\varphi(\tau) = \bar \varphi, \qquad \psi(0) = \bar \psi,
		\end{equation}
		and $\varphi$ and $\psi$ satisfy the following identities in the classical sense on $(0,\tau) \times \T^d$
		\begin{equation}
		\label{eq:hamilton_jacobi_bellman}
		\partial_t \varphi + \frac{1}{2}|\nabla \varphi|^2 + \frac{\eps}{2\tau} \Delta \varphi = 0 \qquad \mbox{and} \qquad \partial_t \psi + \frac{1}{2}|\nabla \psi|^2 - \frac{\eps}{2\tau} \Delta \psi = 0.
		\end{equation}
	\end{Prop}
	
	\begin{proof}[Formal proof]
		The arguments of the previous proof show that if $(\rho,v)$ is optimal in~\eqref{eq:rescaled_sch}, $\varphi$ is an optimizer in~\eqref{eq:dual_formula_varphi} and $\bar \varphi$ is the corresponding optimizer of~\eqref{eq:dual_forward}, then:
		\begin{itemize}
			\item  $v = \nabla \varphi$,
			\item $\varphi$ is a solution of the first equation in~\eqref{eq:hamilton_jacobi_bellman},
			\item the auxiliary function $\theta \coloneq  \exp(\frac{\tau}{\eps}\varphi)$ solves $\partial_t \theta + \frac{\eps}{2\tau}\Delta \theta = 0$,
			\item $\bar \varphi = \varphi(\tau)$.
		\end{itemize}  
		Therefore, the first equations in~\eqref{eq:explicit_varphi_psi},~\eqref{eq:v_is_nabla_varphi},~\eqref{eq:link_a_varphi} and~\eqref{eq:hamilton_jacobi_bellman} are proved. The second equations -- concerning $\psi$ -- are proved in the same way inverting the direction of time. Finally, to prove~\eqref{eq:explicit_rho}, observe that because of~\eqref{eq:link_v_w} and~\eqref{eq:v_is_nabla_varphi}, we have
		\begin{equation*}
		\partial_t \rho + \Div\left(\rho \nabla\left( \frac{\varphi + \psi}{2} \right) \right) = 0 \quad \mbox{and hence}\quad \partial_t \log \rho + \nabla \log \rho \cdot \nabla \left( \frac{\varphi + \psi}{2} \right) + \Delta \left( \frac{\varphi + \psi}{2} \right) = 0.
		\end{equation*}  
		Therefore, using this PDE and~\eqref{eq:hamilton_jacobi_bellman}, a quick computation shows
		\begin{align*}
		\partial_t \left( \varphi - \psi - \frac{\eps}{\tau} \log \rho \right) &= -\frac{1}{2} \left\{ |\nabla\varphi |^2 - |\nabla \psi|^2 \right\} + \frac{\eps}{\tau} \nabla \log \rho \cdot \nabla \left( \frac{\varphi + \psi}{2} \right)\\
		&= - \nabla \left( \frac{\varphi + \psi}{2} \right) \cdot \nabla \left( \varphi - \psi - \frac{\eps}{\tau} \log \rho \right)=0,
		\end{align*}
		where we used that, by~\eqref{eq:link_v_w} and~\eqref{eq:v_is_nabla_varphi}, we have $\nabla(\varphi - \psi - \frac{\eps}{\tau} \log \rho) = 0$. We conclude by observing that at time $t = \tau$, $\varphi(\tau) - \psi(\tau) - \frac{\eps}{\tau} \log \rho(\tau) \equiv 0$ as a consequence of~\eqref{eq:explicit_varphi_psi},~\eqref{eq:link_a_varphi} and~\eqref{eq:phi_from_psi}.
	\end{proof}
	
	As a corollary of this proposition, several evolution PDEs can be derived for relevant quantities of the problem. Here is a list of some that will be useful for us.
	
	\begin{Cor}
		\label{cor:fluid}
		In the situation of Proposition~\ref{prop:link_duality_BB}, calling $\vec m \coloneq  \rho \nabla \varphi$ and $\cev m \coloneq  \rho \nabla \psi$ the so called forward and backward momentums, the following identities hold classicaly in $(0,\tau) \times \T^d$:
		\begin{equation}
		\label{eq:momentum}
		\partial_t \vec m + \bDiv(\vec m \otimes \nabla \psi) + \frac{\eps}{2\tau}\Delta\vec m = 0,
		\qquad
		\partial_t \cev m + \bDiv(\cev m \otimes \nabla \varphi) - \frac{\eps}{2\tau} \Delta\cev m = 0
		\end{equation} 
		(here the notation $\bDiv$ stand for the divergence of a matrix field, which provides a vector field as a result).
		Moreover, denoting by $\vec e = \rho|\nabla \varphi|^2/2$ and $\cev e = \rho|\nabla \psi|^2/2 $ the so-called forward and backward local energies, we have classically in $(0,\tau) \times \T^d$ 
		\begin{equation}
		\label{eq:PDE_kinetic}
		\partial_t \vec e + \Div (\vec e \nabla \psi ) + \frac{\eps}{2\tau} \Delta \vec e = \frac{1}{2} \rho |\DD\varphi|^2, \qquad  \partial_t \cev e + \Div (\cev e \nabla \varphi ) - \frac{\eps}{2\tau} \Delta \cev e = -\frac{1}{2} \rho |\DD\psi|^2.
		\end{equation}
	\end{Cor}
	\begin{proof}
		This proof is only algebraic. Recall that all the following PDEs hold in a classical sense on $(0,\tau) \times \T^d$, and that all the quantities involved are of class $C^\infty$ because of~\eqref{eq:explicit_varphi_psi} and~\eqref{eq:explicit_rho}:
		\begin{equation*}
		\left\{ \begin{gathered}
		\partial_t \rho + \Div(\rho \nabla \varphi) - \frac{\eps}{2\tau} \Delta \rho = 0,\\
		\partial_t \varphi + \frac{1}{2} |\nabla \varphi|^2 + \frac{\eps}{2\tau} \Delta \varphi = 0,
		\end{gathered} \right.
		\qquad  \left\{ \begin{gathered}
		\partial_t \rho + \Div(\rho \nabla \psi) +\frac{\eps}{2\tau} \Delta \rho =0,\\
		\partial_t \psi + \frac{1}{2} |\nabla \psi|^2 -\frac{\eps}{2\tau} \Delta \psi = 0.
		\end{gathered} \right.
		\end{equation*}
		Therefore, we have
		\begin{align*}
		\partial_t \vec m + \bDiv(\vec m \otimes \nabla \psi) &= \partial_t(\rho \nabla \varphi) + \bDiv(\nabla \varphi \otimes (\rho \nabla \psi)) \\
		&= \Big( \partial_t \rho + \Div (\rho \nabla \psi) \Big) \nabla \varphi + \rho \Big( \nabla( \partial_t \varphi) + \DD\varphi \cdot \nabla \psi \Big)\\
		&= - \frac{\eps}{2\tau} \Delta \rho \nabla \varphi+ \rho \Big( \nabla( \partial_t \varphi) + \DD\varphi \cdot \nabla \psi \Big).
		\end{align*}
		Recalling that
		\begin{equation*}
		\nabla \psi = \nabla \varphi - \frac{\eps}{\tau}\nabla \log \rho,
		\end{equation*}
		we deduce that
		\begin{align*}
		\partial_t \vec m + \bDiv(\vec m \otimes \nabla \psi) &=  - \frac{\eps}{2\tau} \Delta \rho \nabla \varphi+\rho \Big( \nabla\Big( \partial_t \varphi + \frac{1}{2} |\nabla \varphi|^2\Big) - \frac{\eps}{\tau}\DD\varphi \cdot \nabla \log \rho \Big)\\
		&= - \frac{\eps}{2\tau}\Big( \Delta \rho \nabla \varphi + 2 \DD \varphi \cdot \nabla \rho + \rho \nabla \Delta \varphi\Big) = - \frac{\eps}{2\tau} \Delta \vec m,
		\end{align*}
		which is the first equation in~\eqref{eq:momentum}. The second one is proved in the same way.
		
		The same computations let us compute:
		\begin{align*}
		\partial_t \vec e + \Div (\vec e \nabla \psi ) &= \frac{1}{2} \Big( \partial_t (\vec m \cdot \nabla \varphi) + \Div \big((\nabla \psi \otimes \vec m) \cdot \nabla \varphi\big) \Big)\\
		&=\frac{1}{2} \Big( \partial_t \vec m + \bDiv(\vec m \otimes \nabla \psi) \Big)\cdot \nabla \varphi +  \vec m\cdot\Big( \nabla (\partial_t \varphi) + \DD \varphi \cdot \nabla \psi \Big)\\
		&= - \frac{\eps}{4\tau}\Delta \vec m \cdot \nabla \varphi + \vec m \cdot \Big( \nabla\Big(  \partial_t \varphi + \frac{1}{2} |\nabla \varphi|^2 \Big) - \frac{\eps}{\tau} \DD \varphi \cdot\nabla \log \rho \Big)\\
		&= - \frac{\eps}{4\tau} \Big( \Delta \vec m \cdot \nabla \varphi + 2 \vec m \cdot (\DD \varphi \cdot \nabla \log \rho) + \vec m \cdot \nabla \Delta \varphi  \Big).
		\end{align*}
		On the other hand,
		\begin{equation*}
		\Delta \vec e = \frac{1}{2}\Delta \vec m \cdot \nabla \varphi + \tr(\DD \varphi\cdot \mathrm D \vec m) + \frac{1}{2}\vec m \cdot \nabla \Delta \varphi.
		\end{equation*}
		Therefore, also using that $\vec m = \rho \nabla \varphi$ and hence $\mathrm D \vec m= \rho \DD \varphi + \nabla \varphi \otimes \nabla \rho$, we find
		\begin{align*}
		\partial_t \vec e + \Div (\vec e \nabla \psi ) + \frac{\eps}{2\tau}\Delta \vec e &= \frac{\eps}{2\tau}\Big( \tr(\DD \varphi \cdot \mathrm D \vec m) - \vec m \cdot(\DD \varphi \cdot \nabla \log \rho) \Big) \\
		&= \frac{\eps}{2\tau}\tr\Big(\DD \varphi \cdot \{ \mathrm D \vec m - \vec m \otimes\nabla \log \rho \}\Big)\\
		&=\frac{\eps}{2\tau}\tr\Big(\DD \varphi \cdot \{ \mathrm D \vec m - \nabla \varphi \otimes \nabla \rho  \}\Big) = \frac{\eps}{2\tau}\rho |\DD \varphi|^2.
		\end{align*}
		This is the first equation in~\eqref{eq:PDE_kinetic}, and the second one follows the same lines.
	\end{proof} 
	
	The final result of this section is a keystone when deriving the optimality conditions of one step of the regularized JKO scheme. It essentially asserts that the subdifferential of the convex functional $\nu \mapsto D_\eps(\mu,\nu)^2$ is the set of maximizers in~\eqref{eq:dual_forward}. We provide a detailed proof of this result.
	
	\begin{Prop}
		\label{prop:right_derivative_sch}
		Let $\mu,\nu\in \P(\T^d)$ with $H(\mu)$ and $H(\nu)$ finite, and let $\eps,\tau>0$. Finally, let $\varphi$ be a maximizer of~\eqref{eq:dual_forward}.
		
		\begin{itemize}
			\item If $\nu'\in \P(\T^d)$ is not absolutely continuous with respect to $\nu$ and $H(\nu')< + \infty$, then calling $\nu_u \coloneq  (1-u) \nu + u \nu'$ for $u\in[0,1]$, we have:
			\begin{equation*}
			\frac{\D }{\D u+}\frac{D_\eps(\mu,\nu_u)^2}{2\tau}\bigg|_{u=0} = -\infty.
			\end{equation*} 
			\item If $\zeta \in L^\infty(\nu)$ and $\int \zeta \D \nu = 0$, then calling $\nu_u \coloneq  (1 + u \zeta) \nu$ for $u \in \R$ with $|u|$ small enough, we have:
			\begin{equation*}
			\frac{\D }{\D u}\frac{D_\eps(\mu,\nu_u)^2}{2\tau}\bigg|_{u=0} = \int \varphi \zeta \nu.
			\end{equation*} 
		\end{itemize}
	\end{Prop}
	\begin{proof}
		\underline{First point}. Let us call $(\rho,w)$ and $(\rho',w')$ the optimizers for problem~\eqref{eq:Sch_backward} between $\mu$ and $\nu$, and $\mu$ and $\nu'$ respectively as given by Proposition~\ref{prop:existence}. By Proposition~\ref{prop:sch_convex} and Corollary~\ref{cor:sym}, we easily see that for all $u\in[0,1]$,
		\begin{equation*}
		\frac{D_\eps(\mu,\nu_u)^2}{2\tau} \leq \frac{\eps}{\tau} H(\nu_u) + (1-u)\left\{ \frac{D_\eps(\mu,\nu)^2}{2\tau} - \frac{\eps}{\tau} H(\nu) \right\} + u\left\{\frac{D_\eps(\mu,\nu')^2}{2\tau} - \frac{\eps}{\tau} H(\nu') \right\}.
		\end{equation*}
		Rearranging the terms and neglecting terms with the right sign, we find for all $u \in (0,1]$:
		\begin{equation*}
		\frac{D_\eps(\mu,\nu_u)^2 - D_\eps(\mu,\nu)^2}{2\tau u} \leq \frac{\eps}{\tau}  \frac{H(\nu_u) - H(\nu)}{u} + \frac{\eps}{\tau}  H(\nu) + \frac{D_\eps(\mu,\nu')^2}{2\tau}.
		\end{equation*} 
		Therefore, we only need to show that
		\begin{equation*}
		\frac{\D}{\D u+}H(\nu_u)\bigg|_{u=0} = - \infty.
		\end{equation*}
		This is an easy consequence of the monotone convergence theorem since 
		\begin{equation*}
		u \longmapsto \frac{\nu_u \log \nu_u - \nu \log \nu}{u}
		\end{equation*}
		is nonincreasing and converges to $-\infty$ on the set $\{ \nu'>0\} \cap \{\nu = 0\}$, which is of positive Lebesgue measure.
		
		\bigskip
		
		\noindent \underline{Second point}. Take $u \in \R$ sufficiently small so that we have $(1 + u\zeta) \nu \in \P(\T^d)$ (i.e. $1+u\zeta\geq 0$). By Proposition~\ref{prop:duality}, we have
		\begin{align*}
		\frac{D_\eps(\mu,\nu_u)^2}{2\tau} &\geq  \frac{\eps}{\tau} H(\mu) + \Big\cg \varphi, \nu_u \Big\cd - \Big\cg \frac{\eps}{\tau} \log \Big( \left\{ \exp\left(\frac{\tau \varphi}{\eps}\right) \right\}\ast \sigma_{\eps} \Big), \mu \Big\cd\\
		&= \frac{D_\eps(\mu,\nu)^2}{2\tau} +  \Big\cg \varphi , \nu_u - \nu \Big\cd = \frac{D_\eps(\mu,\nu)^2}{2\tau} + u \Big\cg \varphi , \zeta \nu \Big\cd.
		\end{align*}
		On the other hand, let us call $\varphi_u$ the optimizer in~\eqref{eq:dual_forward} with $\nu_u$ in place of $\nu$, chosen so that $\int \exp( \frac{\tau\varphi_u}{\eps} )= 1$ (recall that it is unique up to an additive constant). First, let us prove that $\varphi_u$ converges to $\varphi$ as $u \to 0$ up to an additive constant. Indeed, the optimality condition~\eqref{eq:optimality_condition_varphi} for $\varphi_u$ writes:
		\begin{equation}
		\label{eq:opt_varphi_u}
		\varphi_u = \frac{\eps}{\tau} \log \nu_u  - \frac{\eps}{\tau}\log \left( \left\{ \frac{\mu}{\exp\left( \frac{\tau}{\eps} \varphi_u \right)\ast \sigma_\eps} \right\} \ast \sigma_\eps \right).
		\end{equation}
		Yet, $(\exp(\frac{\tau}{\eps} \varphi_u))_{u}$ is precompact in the narrow topology of $\P(\T^d)$, so  $(\exp(\frac{\tau}{\eps} \varphi_u)\ast \sigma_\eps)_{u}$ is precompact in $C(\T^d)$ with a uniform below bound, and so the second term in the right-hand side of~\eqref{eq:opt_varphi_u} is precompact in $C(\T^d)$. Also, the first term in the right hand side of~\eqref{eq:opt_varphi_u} clearly converges to $\frac{\eps}{\tau} \log \nu$ in $L^1(\T^d)$. In turn, $(\varphi_u)_u$ is precompact in $L^1(\T^d)$. Let $\bar \varphi$ be a corresponding limit point. Passing to the limit in~\eqref{eq:opt_varphi_u}, we deduce that $\bar \varphi$ satisfies the optimality condition~\eqref{eq:optimality_condition_varphi} associated with the marginals $(\mu,\nu)$ and by uniqueness, $\varphi - \bar \varphi$ is constant.
		
		Now, the same argument as before provides
		\begin{equation*}
		\frac{D_\eps(\mu,\nu_u)^2}{2\tau} \leq \frac{D_\eps(\mu,\nu)^2}{2\tau} + u \Big\cg \varphi_u , \zeta \nu \Big\cd.
		\end{equation*}
		Using~\eqref{eq:opt_varphi_u} makes it easy to show
		\begin{equation*}
		\Big\cg\varphi_u, \zeta \nu \Big\cd \underset{u \to 0}{\longrightarrow} \Big\cg \bar \varphi , \zeta \nu\Big \cd = \Big\cg  \varphi , \zeta \nu\Big \cd,
		\end{equation*}
		where the last equality follows from $\int \zeta \nu = 0$ and the fact that $\varphi-\bar \varphi$ is constant.
		
		All in all,
		\begin{equation*}
		\frac{D_\eps(\mu,\nu)^2}{2\tau} + u \Big\cg \varphi , \zeta \nu \Big\cd \leq 	\frac{D_\eps(\mu,\nu_u)^2}{2\tau} \leq \frac{D_\eps(\mu,\nu)^2}{2\tau} + u \Big\cg \varphi , \zeta \nu \Big\cd + \underset{u \to 0}{o}(u),
		\end{equation*}
		and the result follows.
	\end{proof}

	\section{Proof of Theorem~\ref{thm:main}}
	\label{sec:proof}
	
	Let $(\bar \rho^k)$ be a piecewise constant curve as in the statement of the theorem, and $T>0$ is a finite time horizon given once for all. For the sake of simplicity, we will assume that for all $k \in \N$, $T$ is an integer multiple of $\tau_k$, but the proof is easy to adapt otherwise. 
    
    It will often be convenient to work with another curve than $\rho^k$ valued in $\P(\T^d)$, interpolating between the sequence defined through the JKO scheme. We define this interpolation as follows. Let $k,n \in \N$ and $(\mu,\xi) = (\mu(s), \xi(s))$, $s \in [0,\tau_k]$ be the optimal pair in~\eqref{eq:rescaled_sch} as given by Proposition~\ref{prop:existence} with $\eps = \eps_k$ and $\tau = \tau_k$, between $\rho^k_n$ and $\rho^k_{n+1}$. For all $s \in (0,\tau_k]$, we define
	\begin{equation*}
	\rho^k(n\tau_k + s) \coloneq  \mu(s) \qquad \mbox{and} \qquad v^k(n \tau_k + s) \coloneq  \xi(s).
	\end{equation*}
	By definition, it is clear that for every $k \in \N$ we have:
	\begin{itemize}
		\item By Proposition~\ref{prop:existence}, the curve $ \rho^k$ belongs to $AC^2([0,T]; \P(\T^d))$ and so in particular is continuous.
		\item Distributionally, on $(0,T) \times \T^d$, recalling that $\alpha_k = \eps_k/\tau_k$,
		\begin{equation}
		\label{eq:evolution_PDE_rhok}
		\partial_t  \rho^k + \Div ( \rho^k  v^k) = \frac{\alpha_k}{2} \Delta  \rho^k.
		\end{equation}
		(This equation is obviously satisfied on sets of type $(n\tau_k, (n+1)\tau_k)$, $n \in \N$, and we deduce that it is satisfied in the whole $(0,T) \times \T^d$ using continuity of $\rho^k$.)
		\item For all $n \in \N$,
		\begin{equation*}
		\rho^k(n\tau_k) = \bar \rho^k(n\tau_k) = \rho^k_n.
		\end{equation*}
	\end{itemize}
	
	From now on, the proof is divided into several steps following quite closely the proof for the classical JKO scheme (see~\cite[Section~8.3]{santambrogio2015optimal} or~\cite[Section~4.4]{santambrogio2017overview}), with two really new ideas. In Steps 1 and 2, we derive and exploit estimates showing that $(\rho^k)_{k \in \N}$ is precompact in the uniform topology of $C([0,T]; \P(\T^d))$. This is made by comparing $\rho^k$ between two time steps with solutions of the heat equation, which is the first novelty of our proof. Then, in Steps 3 to 6, we prove the distributional convergence~\eqref{eq:distributional_convergence}. The main arguments to this convergence are the derivation of the optimality conditions~\eqref{eq:optimality_condition_rhokn} for each step of the problem and a quantitative study of the continuity of the momentums $\rho^k v^k$ in some adapted topology. Finally, Steps 7 to 9 are devoted to proving that every term in~\eqref{eq:distributional_convergence} converges as $k \to 0$. To this aim, the only difficulty is to show the convergence
	\begin{equation*}
	g(\bar \rho^k) \underset{k \to + \infty}{\longrightarrow} g(\rho).
	\end{equation*}
	This is done thanks to the estimate~\eqref{eq:main_compactness_estimate} that we will introduce later, which ensures strong compactness in space for $\rho^k$. This estimate is a consequence of some monotonicity property of the kinetic energy along Schr\"odinger bridges, the second main new ingredient of the proof. We conclude with a classical argument of the Aubin-Lions type that we decided to deduce directly from~\cite{rossi2003tightness}.

	\bigskip
	
	\noindent \underline{Step 1}. A bound for $v^k$ in $L^2(\rho^k)$.
	
	Let us prove that $(\rho^k)$ is precompact in the uniform topology of $C([0,T]; \P(\T^d))$. As in the non-regularized case, this compactness relies on the fact that the kinetic action along one time step is controlled by the dissipation of the energy functional during this time step. The purpose of this step is to derive this control. We will deduce compactness at Step~2 below.
	
	Our bound is obtained by comparing $(\rho^k,v^k)$ to another competitor at each step, and hence it does not use the optimality conditions for the Schr\"odinger problem. Therefore, at this step, we will only use the Propositions~\ref{prop:equivalence} and~\ref{prop:existence}.

	In the classical case, compactness is obtained by considering stationary curves as competitors. Here, this is done by taking solutions of the heat equation.
	For given $k,n \in \N$ and $s \in (0,\tau_k]$, let us call $\mu(s) \coloneq  \rho^k_n \ast \sigma_{\alpha_k s}$ and $\xi(s) \coloneq  0$. Comparing $\rho^k_n \ast \sigma_{\tau_k}$ to $\rho^k_{n+1}$ in the optimization problem~\eqref{eq:entropic_JKO} and then using $(\mu,\xi)$ as a competitor in~\eqref{eq:rescaled_sch} for computing $D_{\eps_k}(\rho^k_n, \rho^k_n \ast\sigma_{\tau_k})$, we find
	\begin{equation}
	\label{eq:first_dissipation_inequality}
	\int_{n\tau_k}^{(n+1)\tau_k} \hspace{-5pt} \int \frac{| v^k|^2}{2} \rho^k \D s + \mathcal F(\rho^k_{n+1}) \leq \mathcal F(\rho^k_n \ast\sigma_{\tau_k}).
	\end{equation}
	(Observe that the initial entropic term from~\eqref{eq:rescaled_sch} is the same on both sides of this inequality, and hence simplifies.)
	
	On the other hand, for all $\rho \in \P(\T^d)$ with $\mathcal F(\rho) < + \infty$ and $t\geq 0$, we have
	\begin{align*}
	\mathcal F(\rho \ast \sigma_t) - \mathcal F(\rho) &\leq \int \{V \ast \sigma_t - V \} \rho +   \int \Big\{ (  W \ast \sigma_t -  W )\ast \rho \Big\} \rho \ast \sigma_t + \int \{ f(\rho \ast \sigma_t) - f(\rho) \} \D x \\
	&\leq \sup \{ V \ast \sigma_t - V \} + \sup \{ ( W  \ast \sigma_t -  W \}
	\end{align*}
	(where we used Jensen's inequality to bound the last term of the first line.) But now, it is clear that for all $C^2$ function $A$ on $\T^d$ (actually $(\Delta A)_+ \in L^\infty$ is enough), there exists $C>0$ such that
	\begin{equation*}
	\sup \{ A \ast \sigma_t - A\} \leq C t, \qquad \mbox{for all } t \geq 0.
	\end{equation*} 
	Plugging this observation into~\eqref{eq:first_dissipation_inequality}, we find $C$ only depending on $V$ and $W$ such that
	\begin{equation}
	\label{eq:second_dissipation_inequality}
	\int_{n\tau_k}^{(n+1)\tau_k} \hspace{-5pt} \int \frac{| v^k|^2}{2}  \rho^k \D s  \leq \mathcal F(\rho^k_n) - \mathcal F(\rho^k_{n+1}) + C \tau_k.
	\end{equation}
	This is the main estimate needed the proof of compactness of $(\rho^k)$. 
	
	Before giving this proof of compactness, let us provide a few direct consequences of~\eqref{eq:second_dissipation_inequality} that will be useful in the sequel. Following Proposition~\ref{prop:existence}, let us introduce
	\begin{equation*}
	c^k \coloneq  v^k - \frac{\alpha_k}{2} \nabla \log \rho^k \qquad \mbox{and} \qquad  w^k \coloneq  v^k - \alpha_k \nabla \log \rho^k.
	\end{equation*} 
	Identity~\eqref{eq:formula_actions} applies in this context, and plugging it in~\eqref{eq:second_dissipation_inequality} for each of these vector fields, we get
	\begin{align*}
	\frac{1}{2}\int_{n\tau_k}^{(n+1)\tau_k} \hspace{-5pt} \int \left\{| c^k|^2 +  \left| \frac{\alpha_k}{2} \nabla \log \rho^k \right|^2 \right\} \rho^k \D s  &\leq \mathcal F(\rho^k_n) - \mathcal F(\rho^k_{n+1})  +  \frac{\alpha_k}{2}  \left( H(\rho^k_n)-  H(\rho^k_{n+1})\right) + C \tau_k,\\
	\int_{n\tau_k}^{(n+1)\tau_k} \hspace{-5pt} \int \frac{| w^k(s)|^2}{2}  \rho^k(s) \D s  &\leq \mathcal F(\rho^k_n) - \mathcal F(\rho^k_{n+1}) + \alpha_k \left( H(\rho^k_n)-  H(\rho^k_{n+1})\right) + C \tau_k.
	\end{align*}
	Summing such inequalities over values of $n$, we get 
	\begin{gather}
	\label{eq:L2bound_v}\mathcal F( \rho^k(T)) + \int_0^T \hspace{-5pt} \int \frac{|v^k|^2}{2} \rho^k \D t \leq \mathcal F(\rho_0) + C T,\\
	\label{eq:L2bound_c} \mathcal F( \rho^k(T)) + \frac{\alpha_k}{2}H( \rho^k(T)) + \frac{1}{2}\int_0^T \hspace{-5pt} \int \left[| c^k|^2 +  \left| \frac{\alpha_k}{2} \nabla \log \rho^k \right|^2 \right] \rho^k \D t \leq \mathcal F(\rho_0) + \frac{\alpha_k}{2} H(\rho_0) + C T,\\
	\label{eq:L2bound_w}\mathcal F( \rho^k(T)) + \alpha_k H( \rho^k(T))  + \int_0^T \hspace{-5pt} \int \frac{|w^k|^2}{2} \rho^k \D t \leq \mathcal F(\rho_0) + \alpha_k H(\rho_0) + C T.
	\end{gather}
    
	\noindent \underline{Step 2}. $( \rho^k)$ is compact in the uniform topology of $C([0,T]; \P(\T^d))$.

	Recalling the link between the metric derivative of $\rho^k$ and the vector field $c^k$ given by Proposition~\ref{prop:existence}, estimate~\eqref{eq:L2bound_c} provides:
	\begin{equation}
	\label{eq:AC2_bound}
	\frac{1}{2} \int_0^T|\dot{ \rho}^k(s)|^2 \D s  \leq \mathcal F(\rho_0) + \frac{\alpha_k}{2}  H(\rho_0)  + C T.
	\end{equation}
	We conclude using the usual arguments that:
	\begin{itemize}
		\item The sequence $( \rho^k)$ is uniformly bounded in $AC^2([0,T]; \P(\T^d))$, and hence relatively compact in $C([0,T]; \P(\T^d))$.
		\item The distance $D(\bar \rho^k(t),  \rho^k(t)) = D( \rho^k (\lceil t/\tau_k\rceil\tau_k),  \rho^k(t))$ vanishes as $k \to + \infty$, uniformly in~$[0,T]$.
		\item Therefore, $(\bar \rho^k)$ is compact in $L^\infty([0,T]; \P(\T^d))$ and has the same limit points as $( \rho^k)$.
	\end{itemize}
	From now on, up to extraction, we can assume that there is a limiting curve in $C(\R_+; \P(\T^d))$ such that both $\bar \rho^k$ and $ \rho^k$ converge to $\rho$ in $\P(\T^d)$, locally uniformly in time.
	
	\bigskip
	
	\noindent \underline{Step 3}. Distributional convergence, first observations.
	
	Here, we prove that $g(\bar \rho^k)$ belongs to $L^1([0,T] \times \T^d)$ and that the convergence~\eqref{eq:distributional_convergence} holds. 
	
	At Step 4 and 5, we will find that for all $k\in \N$, $n \in \N^*$, $\rho^k_n$ is of class $C^2$ and with values in $(d_-,d_+)$, the interior of $\mathrm{Dom}(f)$. Therefore, we will get for free that $g(\bar \rho^k) \in L^1([0,T] \times \T^d)$, and even that for all $t >0$, in the classical sense, using~\eqref{eq:def_g}:
	\begin{equation}
	\label{eq:link_nablag_nablaf}
	\nabla g (\bar \rho^k(t)) = g'(\bar \rho^k(t)) \nabla \bar \rho^k(t) = \bar \rho^k(t) f''(\bar \rho^k(t)) \nabla \bar \rho^k(t) = \bar \rho^k(t) \nabla f'(\bar \rho^k(t)).
	\end{equation}
	So let us temporary admit this regularity result and gather some information necessary to prove~\eqref{eq:distributional_convergence}.
	
	The very first observation is that $\rho^k$ solves~\eqref{eq:evolution_PDE_rhok}, and as $\bar \rho^k - \rho^k$ converge to $0$ in $L^\infty([0,T]; \P(\T^d))$, this convergence also holds distributionally. Therefore, we can replace $\rho^k$ by $\bar \rho^k$ in the $\partial_t$ and $\Delta$ terms of~\eqref{eq:evolution_PDE_rhok} and deduce
	\begin{equation*}
	\partial_t \bar \rho^k + \Div ( \rho^k v^k) - \frac{\alpha_k}{2} \Delta \bar \rho^k \underset{k \to + \infty}{\longrightarrow} 0.
	\end{equation*}
	Therefore, using~\eqref{eq:link_nablag_nablaf}, we conclude that proving~\eqref{eq:distributional_convergence} reduces to prove that distributionally, 
	\begin{equation}
	\label{eq:convergence_momentum}
	\rho^k v^k + \bar \rho^k  \nabla \Big(  V +  W \ast \bar \rho^k  +  f'(\bar \rho^k) \Big) \underset{k \to + \infty}{\longrightarrow} 0.
	\end{equation} 
	
	This is where the optimality conditions of our optimization problem comes into play. At Step~4 and~5, we will see that this condition writes as follows. Up to considering a left continuous version of $v^k$, for all $n \in \N^*$, we have at time~$n \tau_k$: 
	\begin{equation}
	\label{eq:optimality_condition_rough}
	v^k = -\nabla \Big(  V +  W \ast \bar \rho^k  +  f'(\bar \rho^k) \Big).
	\end{equation}
	Therefore, at that times, convergence~\eqref{eq:convergence_momentum} is actually an equality. 
	
	Moreover, let us call
	\begin{equation}
	\label{eq:def_mk}
	\vec m^k \coloneq  \rho^k v^k.
	\end{equation}
	At Step~6, we will show that for a given $n \in \N$ and $t \in (n\tau_k, (n+1)\tau_k]$, the quantity in~\eqref{eq:convergence_momentum} -- which as a consequence of~\eqref{eq:optimality_condition_rough} is nothing but $\vec m^k(t) - \vec m^k((n+1)\tau^k)$ -- can be controlled using~\eqref{eq:momentum}.

	\bigskip

	\noindent \underline{Step 4}. Optimality conditions.
	
	The purpose of this step is to show that for all $k,n \in \N$:
	\begin{enumerate}
		\item The measure $\rho^k_{n+1}$ has $\Leb$-almost everywhere values in $(d_-, d_+)$, the interior of $\mathrm{Dom}(f)$. 
		\item  If $\varphi$ is a maximizer for the dual problem~\eqref{eq:dual_forward} with $\mu = \rho^k_n$, $\nu = \rho^k_{n+1}$, $\eps = \eps_k$ and $\tau = \tau_k$, the function
		\begin{equation}
		\label{eq:optimality_condition_rhokn}
		\zeta \coloneq  \varphi + V + W \ast \rho^k_{n+1} + f'(\rho^k_{n+1}),
		\end{equation}
		which is well defined almost everywhere thanks to the first point, is constant.
	\end{enumerate} 
	The consequences of these two points will be given in the next step. Our proof relies on the inequality
	\begin{equation}
	\label{eq:naive_optimality}
	\frac{D_{\eps_k}(\rho^k_n , \rho^k_{n+1})^2}{2 \tau_k} + \mathcal F(\rho^k_{n+1}) \leq \frac{D_{\eps_k}(\rho^k_n , \rho)^2}{2 \tau_k} + \mathcal F(\rho)
	\end{equation}
	written for well chosen competitors $\rho \in \P(\T^d)$, and computing derivatives as in Proposition~\ref{prop:right_derivative_sch}.

	Let us first prove that $\Leb$-almost everywhere, $\rho^k_{n+1} \in (d_-, d_+)$. We know that
	\begin{equation*}
	\int f(\rho^k_{n+1} ) < + \infty.
	\end{equation*}
	Therefore, almost everywhere, $\rho^k_{n+1} \in \mathrm{Dom}(f)$. We now need to show that under Assumption~\ref{ass:smooth_f}, $\rho^k_{n+1}$ cannot touch the boundaries of $\mathrm{Dom}(f)$ on a set of positive Lebesgue measure. 
	
	To do so, for $u \in [0,1]$, we define $\rho_u \coloneq  (1-u)\rho^k_{n+1} + u \Leb$. Applying~\eqref{eq:naive_optimality} to $\rho_u$, and using the convexity of $\rho \mapsto \int f(\rho)$, we get for all $u\in[0,1]$
	\begin{equation*}
	\frac{D_{\eps_k}(\rho^k_n , \rho^k_{n+1})^2}{2 \tau_k} + \int\hspace{-4pt} \Big(V + \frac{1}{2}W \ast \rho^k_{n+1}\Big)\rho^k_{n+1} \leq \frac{D_{\eps_k}(\rho^k_n , \rho_u)^2}{2 \tau_k} + \int \hspace{-4pt} \Big(V + \frac{1}{2}W \ast \rho_u\Big)\rho_u + u \int \hspace{-4pt} \Big\{ f(1) - f(\rho^k_{n+1}) \Big\}.
	\end{equation*}
	Rearranging the terms, we find for all $u \in (0,1]$:
	\begin{equation*}
	\int \Big\{ f(\rho^k_{n+1}) -f(1) \Big\} + \int \left( V + W\ast \rho^k_{n+1} \right)\! (\rho^n_{k+1} - 1) + \! \underset{u \to 0}{o}(1)\leq \frac{1}{u}\left( \frac{D_{\eps_k}(\rho^k_n , \rho_u)^2}{2 \tau_k} - \frac{D_{\eps_k}(\rho^k_n , \rho^k_{n+1})^2}{2 \tau_k} \right)\! .
	\end{equation*}
	So the right hand side cannot converge to $- \infty$ as $u \to 0$, and because of Proposition~\ref{prop:right_derivative_sch}, we need to have $\Leb \ll \rho^k_{n+1}$, and hence $\rho^k_{n+1}>0$ almost everywhere. We obtain $\rho^k_{n+1}> d_-$ almost everywhere in the particular case where $d_- = 0$.
	
	Now, we apply once again~\eqref{eq:naive_optimality} to $\rho_u$, but this time using the convexity of $\rho \mapsto D_{\eps_k}(\rho^k_n, \rho)^2$. Rearranging the terms in a similar way as what we just did, we find for all $u \in (0,1]$
	\begin{multline*}
	\frac{D_{\eps_k}(\rho^k_n , \rho^k_{n+1})^2-D_{\eps_k}(\rho^k_n , \Leb)^2}{2 \tau_k}   + \int \left( V + W\ast \rho^k_{n+1} \right) (\rho^n_{k+1} - 1) + \underset{u \to 0}{o}(1)\\ \leq \int \frac{ f((1-u) \rho^k_{n+1} + u) - f(\rho^k_{n+1}) }{u}.
	\end{multline*}
	We conclude by monotone convergence that
	\begin{equation*}
	\int f'(\rho^k_{n+1})(1 - \rho^k_{n+1}) > - \infty,
	\end{equation*}
	and with~\eqref{eq:ass_f_smooth}, we conclude that almost everywhere, $\rho^k_{n+1} > d_-$ also when $d_->0$, and $\rho^k_{n+1}< d_+$, even when $d_+<+\infty$.
	
	Let us prove the second point of our claim. For all $\eta>0$ sufficiently small, let us call
	\begin{equation*}
	A_\eta \coloneq  \Big\{ x \in \T^d \mbox{ such that }\rho^k_{n+1}(x) \in (d_- + \eta, \min(d_+ - \eta, 1/\eta) \Big\},
	\end{equation*}
	and let us call $A$ the increasing limit of $A_\eta$, which is of full Lebesgue measure. On $A$, formula~\eqref{eq:optimality_condition_rhokn} defines $\zeta$ properly. For a given $\eta>0$ sufficiently small to have $\rho^k_{n+1}(A_\eta)>0$, let us call
	\begin{equation*}
	\zeta_\eta \coloneq  \left( \zeta - \frac{1}{\rho^k_{n+1}(A_\eta)} \int_{A_\eta}\zeta \rho^k_{n+1}\right)\1_{A_\eta}.
	\end{equation*}
	It is clearly a bounded function (recall that $\varphi - \alpha_k \log \rho^k_{n+1}$ is smooth in virtue of Proposition~\ref{prop:duality}) with $\int \zeta_\eta \rho^k_{n+1} = 0$. Hence, for $u \in \R$ with $|u|$ sufficiently small, $\rho_u \coloneq  (1 + u \zeta_\eta)\rho^{k}_{n+1} \in \P(\T^d)$, so~\eqref{eq:naive_optimality} applies. Moreover, by Proposition~\ref{prop:right_derivative_sch}, we have
	\begin{equation*}
	\frac{\D }{\D u}\frac{D_\eps(\rho^k_n,\rho_u)^2}{2\tau}\bigg|_{u=0} = \int \varphi \zeta_\eta \rho^k_{n+1}.
	\end{equation*}
	On the other hand, the following derivative is straightforward to compute using the dominated convergence theorem:
	\begin{equation*}
	\frac{\D }{\D u} \mathcal F(\rho_u)\bigg|_{u=0} = \int \left( V + W \ast \rho^k_{n+1} + f'(\rho^k_{n+1}) \right)\zeta_\eta \rho^k_{n+1}.
	\end{equation*}
	Therefore, for all $\eta>0$ small enough, we have
	\begin{equation*}
	\int \left( \varphi + V + W \ast \rho^k_{n+1} + f'(\rho^k_{n+1})  \right) \zeta_\eta \rho^k_{n+1}  = 0.
	\end{equation*}
	This can be re-written as $\int \zeta \zeta_\eta \rho^k_{n+1}  = 0$ and, using the fact that $\zeta_\eta$ has zero mean, we can subtract a constant from $\zeta$. Yet, $\zeta_\eta$ vanishes outside $A_\eta$, and on this set it differs from $\zeta$ by a constant: this means that we have $\int \zeta_\eta^2 \rho^k_{n+1}  = 0$, i.e. $\zeta_\eta=0$. 
	Hence, for all $\eta>0$ sufficiently small, $\zeta$ is constant on $A_\eta$. As $A$ is of full Lebesgue measure on the torus because of the first point of this step, we get the result. From now on, we set
	\begin{equation}
	\label{eq:def_ak}
	\varphi^k_{n+1} \coloneq  - \Big(V + W \ast \rho^k_{n+1} + f'(\rho^k_{n+1}) \Big).
	\end{equation}	
	Such a function differs by a constant from any maximizer $\varphi$ of~\eqref{eq:dual_forward}, therefore, by Proposition~\ref{prop:duality}, choosing $\varphi$ such that $\varphi - \alpha_k \log \rho^k_{n+1}$ is of class $C^\infty$, the same holds for $\varphi^k_{n+1}$.

	\bigskip
	
	\noindent \underline{Step 5}. Consequences of~\eqref{eq:optimality_condition_rhokn}: regularity of $\rho^k_n$, $n\in \N^*$ and identity~\eqref{eq:optimality_condition_rough}.
	
	A first consequence of the optimality condition~\eqref{eq:optimality_condition_rhokn} is that for all $k\in \N,n \in \N^*$, $\rho^k_{n}$ and $\varphi^k_n$ are of class $C^2$ and $\rho^k_n$ takes values in a compact subset of $(d_-,d_+)$. Indeed, we already saw that $\varphi^k_n - \alpha_k \log \rho^k_{n}$ is of class $C^\infty$, and by assumptions, both $V$ and $W \ast \rho^k_{n+1}$ are of class $C^2$. Therefore,
	\begin{equation*}
	f'(\rho^k_{n+1}) + \alpha_k \log \rho^k_{n+1}
	\end{equation*}
	is of class $C^2$. In particular it is bounded, and using \eqref{eq:ass_f_smooth} we obtain that $\rho^k_{n+1}$ stays away from both $d_-$ and $d_+$. Moreover, since $f'$ is $C^2$ and nondecreasing on $(d_-,d_+)$, the function $f' + \alpha_k \log$ is a $C^2$ diffeomorphism of $(d_-,d_+)$ into its image, which allows to invert it. Our claim follows.
	
	Following Proposition~\ref{prop:link_duality_BB}, let us define for all $n \in \N$ and $s \in (0,\tau]$:
	\begin{equation}
	\label{eq:def_phik}
	\varphi^k(n\tau + s) \coloneq  \alpha_k \log \left(  \exp\left(  \frac{ \varphi^k_{n+1}}{\alpha_k}\right) \ast \sigma_{\alpha_k(\tau - s)} \right).
	\end{equation}
	As $\varphi^k_{n+1}$ is of class $C^2$, $\varphi^k$ is $C^1$ with respect to time and $C^2$ with respect to space on all intervals of type $(n\tau_k, (n+1)\tau_k]$. Moreover, up to changing $v^k$ on a negligible set of $\R_+ \times \T^d$, we have everywhere for all $k,n\in\N$ and $t \in (n\tau_k, (n+1 )\tau_k]$:
	\begin{equation*}
	v^k(t) = \nabla \varphi^k(t).
	\end{equation*}
	At time $t = (n+1) \tau_k$, we find~\eqref{eq:optimality_condition_rough}.
	\bigskip
	
	\noindent \underline{Step 6}. Proof of~\eqref{eq:convergence_momentum}.
	
	The last step in order to prove the distributional convergence~\eqref{eq:distributional_convergence} is to show~\eqref{eq:convergence_momentum}. Let us define
	\begin{equation*}
	\vec m^k \coloneq  \rho^k \nabla \varphi^k.
	\end{equation*}
	This is the same definition as in~\eqref{eq:def_mk}. By Step~5, it is easy to see that $\vec m^k$ is continuous in time and space on all intervals of type $(n \tau_k, (n+1)\tau_k]$, and by Corollary~\ref{cor:fluid}, calling 
	\begin{equation}
	\label{eq:def_psik}
	\psi^k \coloneq    \varphi^k - \alpha_k \log  \rho^k,
	\end{equation}
	we have in the classical sense on all intervals of type $(n\tau_k, (n+1)\tau_k)$:
	\begin{equation*}
	\partial_t \vec m^k + \bDiv(\vec m^k \otimes \nabla \psi^k) + \frac{\alpha_k}{2} \Delta\vec m^k = 0.
	\end{equation*}
	Therefore, for all $t >0$, at least in a distributional sense,
	\begin{align*}
	\rho^k(t) v^k(t) + &\bar \rho^k(t)  \nabla \Big(  V +  W \ast \bar \rho^k(t)  +  f'(\bar \rho^k(t)) \Big) = \vec m^k(t) - \vec m^k(\lceil t/\tau_k \rceil \tau_k)=-\int_t^{\lceil t/\tau_k \rceil }\partial_t \vec m^k \\
	&=  \bDiv \left( \int_t^{\lceil t/\tau_k \rceil \tau_k} \rho^k(s)\nabla \varphi^k(s)\otimes \nabla  \psi^k(s) \D s \right) + \frac{\alpha_k}{2}\Delta\left(\int_t^{\lceil t/\tau_k \rceil\tau_k}  \rho^k(s) \nabla  \varphi^k(s) \D s\right).
	\end{align*}
	Therefore, to get the convergence we want, it is enough to prove that both 
	\begin{equation*}
	A(t,x) \coloneq  \int_t^{\lceil t/\tau_k \rceil\tau_k} \! \rho^k(s,x)\nabla \varphi^k(s,x)\otimes \nabla  \psi^k(s,x) \D s \quad \mbox{and} \quad B(t,x) \coloneq  \int_t^{\lceil t/\tau_k \rceil\tau_k} \! \rho^k(s,x) \nabla  \varphi^k(s,x) \D s
	\end{equation*}
	converge to $0$ as $k \to + \infty$ in $L^1([0,T] \times \T^d)$. We only treat the case of $A$, because the case of $B$ is similar but easier. We have 
	\begin{align*}
	\| A  \|_{L^1([0,T] \times \T^d)} &\leq \int_0^T \hspace{-5pt}\int_{\lfloor t/\tau_k \rfloor\tau_k}^{\lceil t/\tau_k \rceil\tau_k} \hspace{-5pt} \int |\nabla \varphi^k(s)| |\nabla  \psi^k(s)| \rho^k(s)\D x \D s\D t\\
	&\leq \tau_k \int_0^T  \hspace{-5pt} \int |\nabla \varphi^k(t)| |\nabla  \psi^k(t)| \rho^k(t)\D x \D t\\
	&\leq  \tau_k \sqrt{\displaystyle \int_0^T  \hspace{-5pt} \int |\nabla  \varphi^k(t)|^2  \rho^k(t)\D x \D t} \sqrt{\displaystyle \int_0^T \hspace{-5pt} \int|\nabla  \psi^k(t)|^2  \rho^k(t) \D x \D t}\\
	&\leq \tau_k \sqrt{\mathcal F(\rho_0) + C T} \sqrt{\mathcal F(\rho_0) + \alpha_k H(\rho_0) + C T} \underset{k \to + \infty}{\longrightarrow} 0.
	\end{align*}
	In the last line, we used the bounds~\eqref{eq:L2bound_v} and~\eqref{eq:L2bound_w}, as well as $w^k = \nabla \psi^k$ which is a consequence of Proposition~\ref{prop:link_duality_BB}.
	
	\bigskip
	
	\noindent \underline{Step 7}. Strong compactness: first observation.
	
	By Step~2, $\bar \rho^k \to \rho$ in the uniform topology of $C([0,T]; \P(\T^d))$. This implies the following convergence in the sense of distributions, easy to justify term by term:
	\begin{equation*}
	\partial_t \bar \rho^k - \Div ( \bar \rho^k (\nabla V + \nabla W \ast \bar \rho^k ))  - \frac{\alpha_k}{2} \Delta \bar \rho^k \underset{k \to + \infty}{\longrightarrow} \partial_t  \rho - \Div (  \rho (\nabla V + \nabla W \ast \rho ))  - \frac{\alpha}{2} \Delta \rho.
	\end{equation*}
	Note that because of~\eqref{eq:L2bound_v}, for all $t \in \R_+$, $\mathcal F(\bar \rho^k(t))$ is bounded uniformly in $k$. Using the lower semi-continuity of $\rho \mapsto \int f(\rho) \D x$ and the bounds on $V$ and $W$ and letting $k \to + \infty$, we deduce that for all $t \in \R_+$, $\rho(t)$ is absolutely continuous with respect to the Lebesgue measure, has values in $[d_-,d_+]$, and hence $g(\rho)$ is well defined almost everywhere, with possibly infinite values, see Remark~\ref{rem:assumption_f_smooth}.
	
	Given~\eqref{eq:distributional_convergence}, in order to justify that $\rho$ solves~\eqref{eq:elliptic_reg_PDE}, we need to prove:
	\begin{enumerate}
		\item The function $g(\rho)$ belongs to $L^1([0,T] \times \T^d)$.
		\item The convergence
		\begin{equation*}
		g(\bar \rho^k) \underset{k \to + \infty}{\longrightarrow} g(\rho)
		\end{equation*}
		holds in the sense of distributions (we will actually show strong convergence in $L^1([0,T] \times \T^d))$.
	\end{enumerate} 
	
	\bigskip
	
	\noindent \underline{Step 8}. Strong compactness: main estimate.
	
	The main tool towards this goal is the following estimate: There exists $C>0$ only depending on $V$ and $W$ such that for all $k \in \N$,
	\begin{equation}
	\label{eq:main_compactness_estimate}
	\int_0^T \hspace{-5pt}\int \left\{ |\nabla h(\bar \rho^k) |^2 + \alpha_k^2 |\nabla  \sqrt{\bar \rho^k} |^2 \right\} \D x \D t \leq C \Big( \mathcal F(\rho_0) + \alpha_k H(\rho_0) + T\Big),
	\end{equation}
	where for all $s \in (d_-, d_+)$, 
	\begin{equation}
	\label{eq:def_h}
	h(s) \coloneq  \int_1^s \sqrt r f''(r) \D r.
	\end{equation}
	(Recall that by Assumption~\ref{ass:smooth_f}, we have $d_- < 1 < d_+$, and that by Step~4 and the regularity deduced at Step~5, $\bar \rho^k$ only takes values in $(d_-, d_+)$.)

	To prove this estimate, the starting point is~\eqref{eq:L2bound_w}: what we will do is to bound from below the left-hand side of~\eqref{eq:L2bound_w}, using the decrease of the backward kinetic energy deduced from Corollary~\ref{cor:fluid}. Indeed, thanks to~\eqref{eq:PDE_kinetic} and the regularity proved at Step~5, calling $\cev{e}^k \coloneq  \frac{1}{2} \rho^k|\nabla \psi^k|^2 = \frac{1}{2} \rho^k|w^k|^2 $, we see that for all $n \in \N$,
	\begin{equation*}
	\tau_k \int \frac{|\nabla \psi^k((n+1)\tau_k)|^2}{2}\rho^k((n+1)\tau_k)\leq \int_{n \tau_k}^{(n+1)\tau_k}\hspace{-5pt}\int \frac{|w^k|^2}{2} \rho^k \D t . 
	\end{equation*}
	On the other hand, we know by~\eqref{eq:def_ak},~\eqref{eq:def_phik} and~\eqref{eq:def_psik} that
	\begin{equation*}
	\psi^k((n+1)\tau_k) =  - \Big(V + W \ast \rho^k_{n+1} + f'(\rho^k_{n+1}) + \alpha_k \log \rho^k_{n+1} \Big)+C.
	\end{equation*}
	Plugging this identity in the previous estimate, we find
	\begin{equation*}
	\tau_k \int \frac{\Big|\nabla \Big(V + W \ast \rho^k_{n+1} + f'(\rho^k_{n+1}) + \alpha_k \log \rho^k_{n+1} \Big)\Big|^2}{2}\rho^k_{n+1}\leq \int_{n \tau_k}^{(n+1)\tau_k}\hspace{-5pt}\int \frac{|w^k|^2}{2} \rho^k \D t, 
	\end{equation*}
	and finally, because $\bar \rho^k$ is constant on the time step intervals
	\begin{equation*}
	\int_{n \tau_k}^{(n+1)\tau_k}\hspace{-5pt}	\int \frac{\Big|\nabla \Big(V + W \ast \bar\rho^k + f'(\bar\rho^k) + \alpha_k \log \bar\rho^k \Big)\Big|^2}{2}\bar\rho^k \D t\leq \int_{n \tau_k}^{(n+1)\tau_k}\hspace{-5pt}\int \frac{|w^k|^2}{2} \rho^k \D t, 
	\end{equation*}
	Summing this inequality over $n$ and using~\eqref{eq:L2bound_w}, we get:
	\begin{equation*}
	\int_{0}^{T}\hspace{-5pt}	\int \frac{\Big|\nabla \Big(V + W \ast \bar\rho^k + f'(\bar\rho^k) + \alpha_k \log \bar\rho^k \Big)\Big|^2}{2}\bar\rho^k \D t \leq \mathcal F(\rho_0) + \alpha_k H(\rho_0) + CT.
	\end{equation*}
	Now, using the inequality $ |a|^2 \leq 2|a+b+c|^2 +4 (|b|^2 + |c|^2 )$, we find
	\begin{multline*}
	\int_{0}^{T}\hspace{-5pt}\int\Big|\nabla \Big(f'(\bar\rho^k) + \alpha_k \log \bar\rho^k\Big) \Big|^2 \bar \rho^k \D t \\\leq 4\Big( \mathcal F(\rho_0) + \alpha_k H(\rho_0) + CT \Big) + 4 \int_{0}^{T}\hspace{-5pt}\int \Big\{ |\nabla V|^2 + |\nabla W \ast \bar \rho^k|^2  \Big\}\bar \rho^k \D t.
	\end{multline*}
	To treat the left-hand side, observe that by the regularity obtained at Step~5 and the definition~\eqref{eq:def_h} of $h$, we have 
	\begin{align*}
	|\nabla (f'(\bar \rho^k) + \alpha_k \log \bar \rho^k)|^2 \bar \rho^k &= \left|\left(\sqrt{\bar \rho^k} f''(\bar \rho^k ) + \frac{\alpha_k}{\sqrt{\bar \rho^k}}\right)\nabla \bar \rho^k\right|^2\\
	&= \left(\sqrt{\bar \rho^k} f''(\bar \rho^k ) + \frac{\alpha_k}{\sqrt{\bar \rho^k}}\right)^2 |\nabla \bar \rho^k|^2 \\	
	&\geq \left(\sqrt{\bar \rho^k} f''(\bar \rho^k )\right)^2 |\nabla \bar \rho^k|^2 + \left( \frac{\alpha_k}{\sqrt{\bar \rho^k}}\right)^2 |\nabla \bar \rho^k|^2\\
	&= \left| h'(\bar \rho^k) \nabla \bar \rho^k \right|^2 + \left|  \frac{\alpha_k \nabla \bar \rho^k}{\sqrt{\bar \rho^k}} \right|^2\\	
	& = |\nabla h(\bar \rho^k)|^2 + 2 \alpha_k^2 |\nabla \sqrt{\bar \rho^k}|^2,
	\end{align*}
	where the inequality follows from the convexity of $f$. For the right-hand side, we get our estimate by using uniform bounds for $\nabla V$ and $\nabla W$.
	
	\bigskip
	
	\noindent	\underline{Step 9}. A technical lemma of the Poincar'e type.
	
	Estimate~\eqref{eq:main_compactness_estimate} gives a control of $\nabla h(\bar \rho^k)$ in $L^2([0,T] \times \T^d)$. A crucial point in the coming steps is to deduce a control for $h(\bar \rho^k)$ in some Lebesgue space. We will use the following lemma.
	\begin{Lem}
		\label{lem:poincaré}
		\begin{enumerate}
			\item Let $\rho \in \mathcal P(\T^d) \cap L^1(\T^d)$ with $h(\rho) \in H^1(\T^d)$. Let $p \coloneq  2d/(d-2)$. There exists~$C$ only depending on $d$ and $h$ such that:
			\begin{equation*}
			\| h(\rho)_+ \|_{L^p(\T^d)} \leq C(1 + \| \nabla h(\rho) \|_{L^2(\T^d)}).
			\end{equation*}
			
			\item Let $\mathscr A \subset L^1(\T^d)\cap \P(\T^d)$ be a family of functions such that there exists $s_- \in (d_-, 1)$ and $\delta_- \in (0,1)$ such that for all $\rho \in \mathscr A$
			\begin{equation*}
			\Leb(\rho \geq s_-) \geq \delta_-.
			\end{equation*}
			Then, there is $C$ only depending on $d$, $h$, $s_-$ and $\delta_-$ such that for all $\rho \in \mathscr A$ with $h(\rho) \in H^1(\T^d)$:
			\begin{equation*}
			\| h(\rho)_- \|_{L^p(\T^d)} \leq C(1 + \| \nabla h(\rho) \|_{L^2(\T^d)}).
			\end{equation*}
			
			\item Finally, any uniformly equi-integrable family $\mathscr A \subset L^1(\T^d) \cap \P(\T^d)$ satisfies the condition of the second point.
		\end{enumerate}
	\end{Lem}
	\begin{proof}[Proof of Lemma~\ref{lem:poincaré}]
		Of course, by Sobolev embeddings, it is enough to estimate the $L^2$ norms instead of  the $L^p$ norms of $(h(\rho))_\pm$. The main tool is a small adaptation of~\cite[Exercice~15, Section~5.10]{evans2010partial}: for all $\delta \in (0,1)$, there exists a constant $C = C(\delta) >0$ such that for all $f \in H^1(\T^d)$, if $\Leb(f= 0) \geq \delta$, we have
		\begin{equation}
		\label{eq:Poincare_evans}
		\| f \|_{L^2(\T^d)} \leq C(\delta) \| \nabla f \|_{L^2(\T^d)}.
		\end{equation}
		
		An important observation for the proof of this lemma is that $s \mapsto h(s)$ is nonnegative when $s \geq 1$ and nonpositive when $s \leq 1$.
		
		Let us consider $\rho \in \P(\T^d) \cap L^1(\T^d)$ with $h(\rho) \in H^1(\T^d)$ and prove the first point of the lemma. Let $s_+ \in (1, d_+)$ and $\delta_+ \coloneq  1 - 1/s_+$. We have
		\begin{equation*}
		\Leb( \rho \geq s_+ )   \leq \frac{1}{s_+},
		\end{equation*}
		so that
		\begin{equation*}
		\Leb(h(\rho) \leq h(s_+)) \geq \delta_+.
		\end{equation*}
		We conclude with~\eqref{eq:Poincare_evans} that
		\begin{equation*}
		\| (h(\rho) - h(s_+))_+ \|_{L^2(\T^d)} \leq C(\delta_+) \| \nabla h(\rho)\|_{L^2(\T^d)},
		\end{equation*}
		and hence that
		\begin{equation*}
		\| h(\rho)_+ \|_{L^2} \leq h(s_+) + C(\delta_+) \| \nabla h(\rho)\|_{L^2(\T^d)}.
		\end{equation*}
		
		The second point follows the same lines. Let us assume in addition that $\rho \in \mathscr A$ and let us choose $s_- \in (d_-, 1)$ and $\delta_- \in (0,1)$ as in the statement of the lemma. We find that $\Leb(\rho \geq s_-) \geq \delta_-$. We conclude as before that
		\begin{equation*}
		\| h(\rho)_- \|_{L^2} \leq | h(s_-) | + C(\delta_-) \| \nabla h(\rho) \|_{L^2(\T^d)},
		\end{equation*}
		and the conclusion follows.
		
		Let us finally prove the third point. Let $\mathscr A \subset L^1(\T^d)\cap \P(\T^d)$ be a uniformly equi-integrable set. Let $s_- \in (d_-, 1)$. Let $\delta_- = \delta_-(\mathscr A)>0$ be such that for all $\rho \in \mathscr A$, for all Borelian set $A$, if $\Leb(A) < \delta_-$ then, 
		\begin{equation*}
		\int_A \rho < 1 - s_-.
		\end{equation*}
		For all $\rho \in \mathscr A$,
		\begin{equation*}
		1 = \int \rho \leq s_-  + \int_{\{\rho \geq s_-\}} \rho.
		\end{equation*}
		So
		\begin{equation*}
		\int_{\{\rho \geq s_-\}} \rho \geq 1 - s_-,
		\end{equation*}
		and hence $\Leb(\rho \geq s_-) \geq \delta_-$ by definition of $\delta_-$.
	\end{proof}

	\bigskip
	
	\noindent \underline{Step 10}.  Strong compactness of $(\bar \rho^k)$ in $L^1([0,T] \times \T^d)$.
	
	The purpose of this step is to use the estimate~\eqref{eq:main_compactness_estimate} to show that:
	\begin{equation}
	\label{eq:strong_L1_rho}
	\bar \rho^k \underset{k \to + \infty}{\longrightarrow} \rho \mbox{ in the strong topology of } L^1([0,T] \times \T^d).
	\end{equation}
	We will use extensively the results of~\cite{rossi2003tightness}. The idea is to apply~\cite[Theorem~2]{rossi2003tightness} to show that $(\bar \rho^k)$ is precompact in the topology of convergence in measure of $\mathcal M(0,T; L^1(\T^d))$ (see~\cite[Section~1.1]{rossi2003tightness}). As $(\bar \rho^k)$ is clearly uniformly equi-integrable in $L^1(0,T; L^1(\T^d))$ (simply because for all $k \in \N$ and all $t\in \R_+$, $\| \bar \rho^k(t) \|_{L^1(\T^d)} = 1$), by \cite[Proposition~1.10]{rossi2003tightness}, our claim~\eqref{eq:strong_L1_rho} will follow.
	
	We define a functional $\mathfrak F$ which can take two different forms according to $\alpha>0$ or $\alpha =0$. In the case $\alpha>0$, for all $\mu \in L^1(\T^d)$, let us set
	\begin{equation*}
	\mathfrak F(\mu) \coloneq  \left\{  
	\begin{aligned}
	& \int  |\nabla \sqrt \mu|^2 \D x && \mbox{if } \mu \in \P(\T^d), \, \mbox{ and }\sqrt \mu \in H^1(\T^d),\\
	& + \infty &&\mbox{otherwise,}
	\end{aligned}
	\right.
	\end{equation*}
	while in the case $\alpha=0$ we set
	\begin{equation*}
	\mathfrak F(\mu) \coloneq  \left\{  
	\begin{aligned}
	& \int |\nabla h(\mu)|^2 \D x &&\mbox{if } \mu \in \P(\T^d), \,   \mbox{ and } h(\mu) \in H^1(\T^d),\\
	& + \infty &&\mbox{otherwise}.
	\end{aligned}
	\right.
	\end{equation*}

	We will apply~\cite[Theorem~2]{rossi2003tightness} with this $\mathfrak F$ and $g = D^2$. Let us check the assumptions of the theorem with these choices. For what concerns the estimates, we have on the one hand, in virtue of~\eqref{eq:main_compactness_estimate}:
	\begin{equation*}
	\sup_k \int_0^T \mathfrak F (\bar \rho^k_t) \D t < + \infty.
	\end{equation*}
	We look at bounds in time. The estimate~\eqref{eq:AC2_bound} implies that all the curves $t\mapsto \rho_k(t)$ are H\"older continuous with a constant independent of $k$, more precisely $D^2(\rho_k(t),\rho_k(s))\leq K|t-s|$. Moreover, the curves $t\mapsto \bar\rho_k(t)$ are piecewise constants on intervals of length $\tau_k$ and coincide with the previous curves at the end of these intervals. Hence we obtain 
	\begin{equation*}
	D^2(\bar \rho_k(s), \bar \rho_k(t)) = D^2(\rho_k(\lceil s/\tau_k\rceil \tau_k), \rho_k(\lceil t/\tau_k\rceil \tau_k)) \leq K \Big( \lceil t/\tau_k\rceil - \lceil s/\tau_k\rceil \Big) \tau_k.
	\end{equation*}
	This allows to obtain easily the condition 
	\begin{equation}
	\label{eq:second_estim_savaré_rossi}
	\lim_{\tau \to 0} \sup_k \int_0^{T-\tau} D^2(\bar\rho_k(t), \bar \rho_k(t+\tau))\D t = 0,
	\end{equation}
	since we have 
	\begin{equation*}
	\int_0^{T-\tau} D^2(\bar \rho_k(t), \bar \rho_k(t+\tau))\D t \leq K \tau_k \int_0^{T} \Big( \lceil (t+\tau)/\tau_k\rceil - \lceil t/\tau_k\rceil \Big) \D t
	\end{equation*}
	and the integral in the right hand side equals $\tau T/\tau_k$ if $\tau < \tau_k$ and is bounded by $ T \lceil \tau / \tau_k\rceil$ otherwise. In both case, the right hand side in the equation above is bounded by $2 K(T + \sup_k \tau_k) \tau$ and~\eqref{eq:second_estim_savaré_rossi} follows.
	
	Therefore, in order to meet the assumptions of~\cite[Theorem~2]{rossi2003tightness} and hence to conclude~\eqref{eq:strong_L1_rho}, we just need to show that $\mathfrak F$ has compact sublevels in $L^1(\T^d)$ and that $D^2$ is lower semi-continuous in $L^1(\T^d) \times L^1(\T^d)$ (the compatibility condition~(1.21c) from~\cite{rossi2003tightness} is obvious here). 
	
	Concerning $D^2$, this is obvious, since it is well known that the Wasserstein distance metrizes the weak convergence of measure, which is weaker than the strong $L^1$ topology.

	Concerning $\mathfrak F$, we show that for a given $M>0$, the set
	\begin{equation*}
	\mathscr K_M \coloneq  \left\{  \mu \in L^1(\T^d) \mbox{ such that } \mathfrak F(\mu) \leq M \right\}
	\end{equation*}
	is a compact set of $L^1(\T^d)$. 
	
	\bigskip
	
	When $\alpha>0$, let $(\mu_n)_{n \in \N}$ be a sequence in $\mathscr K_M$ and for all $n$, $a_n \coloneq  \sqrt{\mu_n}$. We have $\int a_n^2 \D x = 1$ and $\int |\nabla a_n|^2 \D x \leq M$. Therefore, $(a_n)_{n \in N}$ is bounded in $H^1(\T^d)$, hence precompact in $L^2$, and therefore $(\mu_n)_{n \in \N}$ is precompact in $L^1(\T^d)$. Using classical semi-continuity arguments, we see that it limit points are in $\mathscr K_M$ and the conclusion follows.
	
	\bigskip
	
	When $\alpha=0$, this is slightly more difficult and we have to use condition~\eqref{eq:growth_condition} as well as Lemma~\ref{lem:poincaré}. We consider once again $(\mu_n)_{n \in \N}$ a sequence in~$\mathscr K_M$ and we show that it has a limit point in $\mathscr K_M$ in the strong topology of $L^1$. For all $n$, let $a_n \coloneq  h(\rho_n)$. 
	
	We first prove that $(a_n)$ is bounded in $L^p(\T^d)$ with $p = 2d/(d-2)$. By the first point of Lemma~\ref{lem:poincaré}, $((a_n)_+)$ is bounded in $L^p(\T^d)$. Then, we deduce that $\mathscr K_M$ is uniformly equi-integrable. Indeed, whenever $K > 1$, using~\eqref{eq:growth_condition},
	\begin{equation*}
	\int_{\mu_n\geq K} \mu_n\D x = \int_{\mu_n \geq K} \frac{\mu_n}{h(\mu_n)^p} h(\mu_n)^p \D x \leq \| (a_n)_+ \|_{L^p(\T^d)}^p\sup_{s \geq K} \frac{s}{h(s)^p} \underset{K \to +\infty}{\longrightarrow} 0.
	\end{equation*}
	Therefore, the second and third point of the lemma apply with $\mathscr A = \mathscr K_M$, so we find that $(a_n)$ is bounded in $L^p(\T^d)$ and hence in $H^1(\T^d)$.
	
	Thus, up to considering a subsequence, we have $ h(\mu_n) \to  h(\mu)$ in $L^2(\T^d)$, and so this convergence also holds in measure. But $h^{-1}: \R \to \R_+$ is continuous, so $\mu_n \to \mu$ in measure as well. As in addition, as we already saw, $\mathscr K_M$ is uniformly equi-integrable, $\mu_n \to \mu$ in $L^1(\T^d)$. Once again, the fact that the limit points belong to $\mathscr K_M$ follow from standard semi-continuity arguments.
	
	\bigskip	
	
	\noindent \underline{Step 11}. Strong compactness of $g(\bar \rho^k)$ in $L^1([0,T] \times \T^d)$. 
	
	We can finally conclude the proof by showing that
	\begin{equation*}
	g(\bar \rho^k) \underset{k \to + \infty}{\longrightarrow} g(\rho) \quad \mbox{in} \quad L^1([0,T] \times \T^d).
	\end{equation*}
	By the previous step, $\bar \rho^k \to \rho$ in the strong topology of $L^1([0,T] \times \T^d)$. On the other hand, $g$ is continuous. Therefore, if we can show that $(g(\bar \rho^k))$ is uniformly equi-integrable on $[0,T] \times \T^d$, then we can conclude that $g(\bar \rho^k) \to g(\rho)$ in $L^1([0,T] \times \T^d)$.

	This uniform equi-integrability is obtained showing that both the positive and negative parts of $(g(\bar \rho^k) - g(1))$ are uniformly integrable. Observe that $s \mapsto g(s) - g(1)$ is nonnegative when $s \geq 1$ and nonpositive when $s \leq 1$. First, we have for all $s \geq 1$:
	\begin{equation*}
	g(s) = g(1) + \int_1^s \sqrt r h'(r) \D r \leq g(1) + \sqrt s h(s).
	\end{equation*}
	We deduce that
	\begin{equation*}
	(g(\bar \rho^k) - g(1))_+ \leq |g(1)| +  \sqrt{\bar \rho^k} h(\bar \rho^k)_+.
	\end{equation*}
	But $(\sqrt{\bar \rho^k})$ is bounded in $L^\infty(\R_+; L^2(\T^d))$ and by~\eqref{eq:main_compactness_estimate} and the first point of Lemma~\ref{lem:poincaré}, $(h(\bar \rho^k)_+)$ is bounded in $L^2([0,T];L^p(\R^d))$ for $p = 2d/(d-2)$. So $(g(\bar \rho^k)_+)$ is bounded in $L^2([0,T];L^{p'}(\T^d))$ for some $p' >1$ and uniform equi-integrability for the positive part of $(g(\bar \rho^k))$ on $[0,T] \times \T^d$ follows.
	
	It remains to prove uniform equi-integrability of the negative part of $(g(\bar \rho^k))$. This is where we need to use the second point of Assumption~\ref{ass:strong_compactness}. The easiest case is when $d_- = 0$ and $f(0) < + \infty$. Indeed, in this case, by convexity of $f$, for all $s \in [0,1]$,
	\begin{equation*}
	-g(s) = f(s) - sf'(s) \leq f(0),
	\end{equation*}
	so that for all $k$
	\begin{equation*}
	(g(\bar \rho^k) - g(1))_- \leq |g(1)| + |f(0)|.
	\end{equation*}
	Therefore, $((g(\bar \rho^k) - g(1))_-)$ is bounded in $L^\infty(\R_+ \times \T^d)$ and uniform equi-integrability follows.
	
	In all the other cases, we will use the following inequality: if $s<1$, 
	\begin{equation*}
	g(1) - g(s) =  \int_s^1 \sqrt{r} h'(r) \D r \leq h(s)_-.
	\end{equation*}
	So we need to prove uniform equi-integrability for $(h(\bar \rho^k)_-)$. 
	
	We actually prove that $(h(\bar \rho^k)_-)$ is bounded in $L^2([0,T]; L^2(\T^d))$ using~\eqref{eq:L2bound_c} and the second and third points of Lemma~\ref{lem:poincaré} with a well chosen $\mathscr A$. First, let us fix $K$ a large constant bounding from above the right-hand side of~\eqref{eq:L2bound_c} uniformly in~$k$. We call 
	\begin{equation*}
	\mathscr A \coloneq  \left\{  
	\begin{aligned}
	&\left\{ \rho \in L^1(\T^d) \cap \P(\T^d) \mbox{ such that } H(\rho) \leq 2 K/ (\inf_k \alpha_k)  \right\}, &&\mbox{if } \alpha >0, \\
	&\left\{ \rho \in L^1(\T^d) \cap \P(\T^d) \mbox{ such that } \int f(\rho) \leq K + \|V\|_\infty + \|W \|_\infty \right\}, &&\mbox{if }\alpha =0.
	\end{aligned}
	\right.
	\end{equation*}
	If $\alpha >0$ or if $\alpha = 0$ but $f$ is superlinear, $\mathscr A$ is uniformly equi-integrable and hence, according to the third point of Lemma~\ref{lem:poincaré}, $\mathscr A$ satisfies the second point of Lemma~\ref{lem:poincaré}. If $\alpha = 0$ but $f(d_-) = +\infty$, let us call $\tilde f : s \mapsto f(s) - f(1) - (s-1)f'(1) \in \R_+$. Then for all $s_- \in (d_-, 1)$ and all $\rho \in \mathscr A$,
	\begin{equation*}
	\Leb(\rho < s_-) \tilde f(s_-) \leq \int \tilde f(\rho) = \int f(\rho) - f(1) \leq K + \|V\|_{\infty} + \| W \|_\infty + |f(1)|.
	\end{equation*}
	Therefore, for $s_-$ sufficiently close to $d_-$, $\Leb(\rho < s_-) \leq 1/2$ uniformly in $\rho \in \mathscr A$. In that case as well, $\mathscr A$ satisfies the condition of the second point of Lemma~\ref{lem:poincaré}. 
	
	We conclude that in all these cases, there exists $C>0$ such that for all $\rho \in \mathscr A$ such that $h(\rho) \in H^1(\T^d)$,
	\begin{equation*}
	\| h(\rho)_- \|_{L^2(\T^d)} \leq C(1 + \| \nabla h(\rho) \|_{L^2(\T^d)}).
	\end{equation*}
	Using~\eqref{eq:L2bound_c} and~\eqref{eq:main_compactness_estimate}, we find that $(\nabla h(\bar \rho^k))$ is bounded in $L^2([0,T]; L^2(\T^d))$, and that for all $k$ and almost all $t \in [0,T]$, $\bar \rho^k(t) \in \mathscr A$. Therefore, by the previous inequality, $(h(\bar \rho^k)_-)$ is bounded in $L^2([0,T]; L^2(\T^d))$. This last observation concludes the proof of Theorem~\ref{thm:main}.\qed
	\bigskip
	
\noindent {\bf Acknowledgments.} The authors acknowledge the support of European union via the ERC AdG 101054420 EYAWKAJKOS. 	
	\bibliography{biblio}
	\bibliographystyle{plain}

\end{document}